\documentclass[a4, 11pt,leqno]{amsart}
\usepackage{lmodern}
\usepackage[english]{babel}
\usepackage{microtype}
\usepackage{pifont}
\usepackage{mathtools}
\usepackage[active]{srcltx}
\usepackage[utf8]{inputenc}
\usepackage[T1]{fontenc}
\usepackage{lipsum}
\usepackage{upgreek}
\usepackage{listings}
\usepackage{wasysym}
\DeclareMathAlphabet{\mathdutchcal}{U}{dutchcal}{m}{n}
\SetMathAlphabet{\mathdutchcal}{bold}{U}{dutchcal}{b}{n}
\DeclareMathAlphabet{\mathdutchbcal}{U}{dutchcal}{b}{n}
\DeclareSymbolFont{myletters}{OML}{ztmcm}{m}{it}
\DeclareMathSymbol{\nicelambda}{\mathord}{myletters}{"15}
\usepackage{amsmath,amssymb,amscd,amsthm,graphicx,enumerate}
\usepackage{mathrsfs}
\usepackage{amssymb,amsthm,a4wide}
\usepackage{nicefrac}
\newcounter{example}
\newenvironment{example}[1]{\refstepcounter{example}\par\medskip
	\noindent \textsc{\small Example~\theexample. #1} \rmfamily\hspace{-2pt}}{\medskip}
\usepackage{pifont}

\usepackage{color}
\usepackage[usenames,dvipsnames]{xcolor}
\usepackage{xcolor}
\definecolor{ultrablue}{rgb}{0.0,0.0, 1}
\usepackage{hyperref}
\hypersetup{
	linktoc=page,
	colorlinks,
	linkcolor={ultrablue},
	citecolor={ultrablue},
	urlcolor={ultrablue}
}
\hypersetup{breaklinks=true}
\allowdisplaybreaks
\usepackage{tikz}
\usetikzlibrary{matrix}
\usepackage[all]{xy}
\usepackage{subfig}
\makeatletter 
\def\l@subsection{\@tocline{2}{0pt}{3pc}{6pc}{}}
\makeatother
\usepackage{enumerate}
\usepackage{soul}
\usepackage{graphicx}
\usepackage{tikz}
\usepackage{tikzpagenodes}
\usepackage{scrlayer-scrpage}
\usetikzlibrary{calc,spy}
\usetikzlibrary{shapes.geometric, arrows}
\usepackage{pgfplots}
\usetikzlibrary{intersections, pgfplots.fillbetween}
\makeatletter
\@namedef{subjclassname@2020}{%
	\textup{2020} Mathematics Subject Classification}
\makeatother
\makeatletter 
\def\l@subsection{\@tocline{2}{0pt}{3pc}{6pc}{}}
 \makeatother
\usepackage[a4paper,bindingoffset=0in,%
left=1in,right=1in,top=1.1in,bottom=1.1in,%
footskip=.2in]{geometry}
\theoremstyle{plain}
\newtheorem*{theorem*}{Theorem}
\newtheorem{lemma}{Lemma}[section]
\newtheorem{theorem}[lemma]{Theorem}
\newtheorem{proposition}[lemma]{Proposition}
\newtheorem{corollary}[lemma]{Corollary}

\theoremstyle{definition}
\newtheorem{definition}[lemma]{Definition}
\newtheorem{remark}[lemma]{Remark}
\theoremstyle{remark}

\numberwithin{equation}{section}
\newcommand{\R}{\mathbb R}
\newcommand{\Z}{\mathbb Z}

\newcommand{\Deg}{\mathrm{Deg}}

\DeclareMathOperator{\Ric}{Ric}
\DeclareMathOperator{\scal}{scal}

\DeclareMathOperator{\Tan}{Tan}
\DeclareMathOperator{\Dir}{Dir}

\newcommand{\K}{\mathcal K}

\DeclareMathOperator{\diam}{diam}		% Diam 

\DeclareMathOperator{\Lip}{Lip}
\DeclareMathOperator{\lip}{lip}
\newcommand{\Was}{\mathcal{W}}
\newcommand{\m}{\mathdutchcal{m}}

\newcommand{\G}{\mathcal{G}}
\newcommand{\B}{\mathcal{B}}
\newcommand{\p}{\mathdutchcal{p}}
\newcommand{\h}{\mathdutchcal{h}}

\newcommand{\Prob}{\mathscr{P}}
\newcommand{\LL}{\mathrm{L}}

\newcommand{\olric}{{^{\mathcal O}\!\Ric}}
\newcommand{\oscal}{{^{\mathcal O}\!\scal}}

\newcommand{\dist}{\mathrm{d}}

\newcommand{\cart}{\hspace{1pt}{\text{\scalebox{0.6}{$\square$}}}\hspace{1pt}}
\newcount\stylenum  \newdimen\styledim 
\def\varstyle#1{\mathchoice{\stylenum=0 #1}{\stylenum=1 #1}{\stylenum=2 #1}{\stylenum=3 #1}}
\def\mathaxis{\fontdimen22\ifcase\stylenum 
	\textfont\or\textfont\or\scriptfont\or\scriptscriptfont\fi2 }
\def\setstyledim{\styledim=\ifcase\stylenum .1em\or.1em\or.07em\or.05em\fi\relax}

\def\sqdot{\mathbin{\varstyle{\raise\mathaxis\hbox{\setstyledim
				\kern\styledim 
				\vrule width1.2\styledim height.6\styledim depth.6\styledim
				\kern\styledim}}}}
\newcommand{\restr}{\raisebox{-.1908ex}{$\big|$}}

\begin{document}
% **********************************************************************************
%	Title
% **********************************************************************************
\title[\tiny Discrete Ollivier-Ricci curvature]{ \small  Discrete Ollivier-Ricci curvature}  
% **********************************************************************************
%	Authors
% **********************************************************************************
%\author[Zohreh Fathi]{Zohreh Fathi}
\address{  Department of Mathematics and Computer Science, Amirkabir University of Technology, 424 hafez Ave., Tehran, Iran}
\email{\href{mailto:z.fathi@aut.ac.ir}{z.fathi@aut.ac.ir}}
%\author[Sajjad Lakzian]{Sajjad Lakzian}
\address{ Department of Mathematical Sciences\\ Isfahan University of  Technology (IUT) \\ Isfahan 8415683111, Iran}
\email{\href{mailto:slakzian@iut.ac.ir}{slakzian@iut.ac.ir}}
\address{School of Mathematics\\
	Institute for Research in Fundamental Sciences (IPM) \\
	P.O. Box 19395-5746\\
	Iran
}
\email{\href{mailto:slakzian@ipm.ir}{slakzian@ipm.ir}}
% **********************************************************************************
%	First page footer stuff : MSC-keywords, thanks, etc.
% **********************************************************************************
\keywords{weighted graphs, random walks, Ollivier-Ricci curvature, Ricci flow, linear programming, \phantom{\hspace{20pt}} polytope, upper bound theorem, infinitesimal generator, Markov semigroups}
\thanks{*\textit{the corresponding author}}
% **********************************************************************************
%	make title
% **********************************************************************************
\maketitle
% **********************************************************************************
%	Authors info on first page
% **********************************************************************************
\begin{center}
\begin{minipage}[t]{0.5\textwidth}
	\centering
\bf	{ \small Zohreh Fathi}\\
\textit{\footnotesize Amirkabir University of Technology}\\
\textit{\footnotesize Tehran, Iran}\\
%\textit{\tiny slakzian@ipm.ir,  slakzian@iut.ac.ir}\\
%\textit{\tiny slakzian@iut.ac.ir}
		\vspace{10mm}	
\end{minipage}
\begin{minipage}[t]{0.5\textwidth}
	\centering
	\bf	{ \small Sajjad Lakzian*}\\
	\textit{\footnotesize Isfahan University of Technology}\\
	\textit{\footnotesize Isfahan, Iran}\\
	%\textit{\tiny slakzian@ipm.ir,  slakzian@iut.ac.ir}\\
	%\textit{\tiny slakzian@iut.ac.ir}
\end{minipage}
\end{center}
\vspace{5mm}
%	\begin{minipage}[t]{0.4\textwidth}
%		\centering
%		\bf	{ \small Sajjad Lakzian}\\
%		\textit{\tiny IPM \& Isfahan University of Technology}\\
%		\textit{ \tiny Isfahan, Iran}\\
%		%\textit{\tiny slakzian@ipm.ir,  slakzian@iut.ac.ir}\\
%		%\textit{\tiny slakzian@iut.ac.ir}
%		
%	\end{minipage}
% **********************************************************************************
%	Abstract
% **********************************************************************************
\begin{abstract}
	We analyze both continuous and discrete-time Ollivier-Ricci curvatures of locally-finite weighted graphs $\G$ equipped with a given distance ``$\dist$'' (w.r.t. which $\G$ is metrically complete) and for general random walks. We show the continuous-time Ollivier-Ricci curvature is well-defined for a large class of Markovian and non-Markovian random walks and provide a criterion for existence of continuous-time Ollivier-Ricci curvature; the said results generalize the previous rather limited constructions in the literature. 
\par  In addition, important properties of both discrete-time and continuous-time Ollivier-Ricci curvatures are obtained including -- to name a few -- Lipschitz continuity, concavity properties, piece-wise regularity (piece-wise linearity in the case of linear walks) for the discrete-time Ollivier-Ricci as well as Lipschitz continuity and limit-free formulation for the continuous-time Ollivier-Ricci.  these properties were previously known only for very specific distances and very specific random walks. As an application of Lipschitz continuity, we obtain existence and uniqueness of generalized continuous-time Ollivier-Ricci curvature flows.
\par Along the way, we obtain -- by optimizing  McMullen's upper bounds -- a sharp upper bound estimate on the number of vertices of a convex polytope in terms of number of its facets and the ambient dimension, which might be of independent interest in convex geometry. The said upper bound allows us to bound the number of polynomial pieces of the discrete-time Ollivier-Ricci curvature as a function of time in the time-polynomial  random walk.  The limit-free formulation we establish allows us to define an operator theoretic Ollivier-Ricci curvature which is a non-linear concave functional on suitable operator spaces.
\end{abstract}
% **********************************************************************************
%	Date
% **********************************************************************************
\date{\today}
% **********************************************************************************
%	TOC
% **********************************************************************************
\tableofcontents
% **********************************************************************************
%	Here the body of text starts!
% **********************************************************************************
%------------------------------------------------------------------------------------
%------------------------------------------------------------------------------------
%SECTION: INTRODUCTION
%------------------------------------------------------------------------------------
%------------------------------------------------------------------------------------
\section{Introduction}
\par Along with the vivid rise in using data and network analysis, the scientific community has witnessed emergence of many important discrete models for example in life sciences and finance; study of these models -- beyond the classical ways -- would require new novel techniques. 
 \par Also as the research in discrete structures furthers, it further reveals the innate power that lies in  the seeming ``reduction'' that takes place combined with the ``adequacy'' that is retained when approximating continuous structures by discrete ones; meaning, the theories and computations become programmable (reduction) and yet the recent developments in discrete geometry indicate that one can still successfully apply classical ideas and tools to fruition (adequacy). 
 \par Because of the said revelations, and inspired by the already established power of geometric-analytic tools in continuous phase spaces, the interest in applying such methods to discrete structures has surged among both Mathematicians and scientists from other fields alike. Needless to say, a nontrivial geometry entails the introduction of curvature; and the quest of studying  various types of  curvature of weighted graphs -- as epitome of discrete structures -- has proven very fruitful; many similarities to the continuous setting has been unearthed and many useful generalizations made. 
\par Unlike what we see in the Riemannian manifolds, when it comes to weighted graphs, there are numerous ways of defining a notion of ``Ricci curvature''. Some main stream ones include Ricci curvature bounds using optimal transport theory (the so called Lott-Sturm-Villani curvature bounds)~\cite{BS}, discrete Bakry-\'Emery bounds~\cite{LY}, Ollivier-Ricci curvature~\cite{Ol, LLY} and Forman-Ricci curvature~\cite{For}, each suited for a set of different purposes. From the geometrical point of view, if we look at a network as the 1-skeleton of a CW-complex, the Bakry-\'Emery curvature-dimension bounds are construed as vertex type curvatures or curvatures on the 0-skeleton.  In these notes, we wish to look at the  Ollivier's definition of coarse Ricci curvature on the 1-skeleton. 
\par The aforementioned Ricci curvatures appear in the study of many real life discrete models. These applications -- discovered not so long a go -- include the use of Ollivier-Ricci and Bakry-\'Emery Ricci curvature as indicators of robustness and as tools by which to measure the difference of two networks; the applications have thus far been in the fields of social, biological or financial networks and at a growing rate due to successes achieved by using these methods. For more details, we refer to~\cite{FL} and the reference therein.  The other important application of Ollivier-Ricci curvature and flow is in community detection i.e. finding clusters with high connectivity in weighted networks \cite{NLLG, SJB, JL}. 
\par We are considering the original version of Ollivier-Ricci curvature instead of limited modified versions, an endeavor that is long overdue. The already existing constructions and modified versions of Ollivier-Ricci curvature are very useful yet they are not exploring the full potential of  Ollivier's definition of coarse Ricci curvature. It is worth mentioning that almost all of the already established properties for the modified versions of Ollivier-Ricci curvature would follow as spacial cases of our results if we restrict ourselves to $\upbeta$-walks (lazy walks); see~\textsection\thinspace\ref{sec:theta-walks}. 
\par  This article is inspired by~\cite{LLY, MW} and also -- to a large extent -- generalizes many constructions in the said works.
 %-----------------------------------------------
%-----------------------------------------------
%-----------------------------------------------
 \subsection*{A brief setup and notations}
Consider a quadruple $\left( \G, \m, \omega, \dist \right)$ where $\G$ is a locally-finite graph, $\m$ is a vertex  measure, $\omega = \left\{  \omega_i  \right\}_{i \in I}$  is a finite set of edge weight functions i.e. we can consider $\omega: E \to \R^I$ as a vector-valued edge weight, and $\dist$ is a distance on $\G$ with the condition that $\left( \G, \dist \right)$ is complete, thus, a Polish metric space. For defining the Laplacian, we will just use  the notation $\omega$ which can be considered as the zeroth element ($\omega_0$)  in the collection of the said weight functions. Sometimes, we will also allow the distance $\dist$ to be induced by a secondary edge weight $\eta$. 
\par A random walk on $\G$ is a collection of probability measures $\upmu_z, z\in \G$; $\upmu_z^\varepsilon$ determines -- once at $z$ -- where and with what probability, a walker can go next. In these notes, we will work with a  $1$-parameter family of random walks $\upmu_z^\varepsilon$ for $\varepsilon \in [0,1]$. We will be concerned with $1$-parameter walks that are continuous in $\varepsilon$ and with $\upmu_z^0 = \delta_z$, so this will be an standing assumption throughout. 
\par For a given fixed $\varepsilon$,  the discrete-time Ollivier-Ricci curvature is given (a la Ollivier) by 
  \[
  \olric_\varepsilon (x,y) := 1 - \nicefrac{\Was_1\left( \upmu_x^\varepsilon ,  \upmu_y^\varepsilon  \right)}{\dist(x,y)}, \quad 0\le \varepsilon \le 1.
  \]
where $\Was_1$ is the $\LL^1$-Wasserstein distance also known as the Kantorovich-Rubinstein metric.   
 \par Both the discrete-time and continuous-time Ollivier-Ricci curvatures are formulated using this $1$-parameter family of random walks. Of course in the discrete-time version the parameter does not play a role in the definition however it becomes important when one considers the discrete-time Ollivier-Ricci as a function of the parameter. 
 \begin{definition}[finite-step walks]\label{defn:fin-step}
   	A $1$-parameter family $\upmu_z, z\in \G$; $\upmu_z^\varepsilon$ of random walks is said to be a \emph{finite-step} walk when for each $z$ and each $\varepsilon$,  it has bounded support w.r.t. the combinatorial distance which in conjunction with local finiteness means the support is a finite set. 
 \end{definition} 
  \begin{definition}[time-affine, time-polynomial or time-analytic walks] A continuous-time random walk $\upmu_z^\varepsilon$ is said to be time-affine, time-polynomial or time-analytic whenever for every fixed $x,y$, $\p_{xy}(\varepsilon):= \upmu^\varepsilon_x(y)$ is affine, polynomial or analytic (resp.) in $\varepsilon$ and \emph{admits an analytic continuation} over  $\left( -\updelta_{xy}, 1+\updelta_{xy} \right)$  for some small $\updelta_{xy}>0$.  
  	\par Notice the  analytic continuation condition is automatically satisfied by time-affine and time-polynomial walks. Also notice we are not assuming the walk has a continuation, so the continuation might not be a nonnegative measure or a probability measure. 
  \end{definition}
We should remark that when working with the continuous-time Ollivier-Ricci curvature, we only need the walk to admit a continuation beyond $\varepsilon=0$. 
\begin{definition}\label{defn:loc-exc}
  	A continuous-time random walk $\upmu_z^\varepsilon$ is said to be a local walk whenever for each $z$, there exists a finite subgraph ${\mathcal K}_{z}$  such that $\Upomega^\varepsilon_{z} =  \mathrm{supp}\left( \upmu_z^\varepsilon \right) $ is included in ${\mathcal K}_{z}$ for all $\varepsilon$. 
  	This means for fixed $x,y$, we have $\Upomega^\varepsilon_{xy} := \mathrm{supp}\left( \upmu_x^\varepsilon \right) \cup \mathrm{supp}\left( \upmu_y^\varepsilon \right)$ is included in a finite set $\K_{xy}$ with cardinality ${\mathcal N}_{xy}:=|\mathcal{K}_{xy}|$. In particular, a local walk is a finite-step walk but the converse does not hold. We set ${\mathcal N}_{xy}:=|\mathcal{K}_{xy}|$.
  \end{definition}
 \par Special cases of walks that our results apply to, include $\uptheta$-walks; in particular, if one uses the restricted cases of $\upbeta$, $\upzeta$ and $\upxi$-walks (these are time-affine and are also called lazy walks in the literature), one retrieves the known constructions in the literature~\cite{LLY,MW,JM, NLLG};  see~\textsection\thinspace\ref{sec:theta-walks} for further details and definitions. 
  %-----------------------------------------------
  %-----------------------------------------------
   %-----------------------------------------------
 \subsection*{Summary of Main  results}
 Here we give a brief mention of the definitions and main results; further details and proofs are to be found in the referenced sections. 
 %-----------------------------------------------
 %-----------------------------------------------
\subsubsection*{\small \bf \textsf{Discrete-time Ollivier-Ricci curvature}}
For more details regarding the following theorem, see~\textsection\thinspace\ref{sec:dt-olric}. Notice that the discrete-time Ollivier-Ricci is well-defined as soon as the walk had finite first moment; for local walks, we get much more.
\begin{theorem}
	Suppose $\upmu_z^\varepsilon$ is a (time-affine or time-polynomial)  time-analytic local random walk. For any fixed pair $x,y$, $\olric_\varepsilon (x,y)$ as a function of $\varepsilon$, satisfies the following;
\begin{enumerate}
	\item (piece-wise regularity) it is a piece-wise (affine, polynomial) analytic function of $\varepsilon$;
	\item (locality) it coincides with the discrete-time curvature computed in the subgraph $\Upomega^\varepsilon_{xy}$;
	\item (concavity for time-affine) it is concave for time-affine walks (also holds without the locality condition);
	\item (Lipschitz regularity) it is Lipschitz continuous in $\varepsilon$; 
	\item (finite pieces) it admits finitely many regular pieces;
	\item (upper bound for time-polynomial walks) for time polynomial walks, the number of distinct polynomial pieces  is bounded above by
	\[
		\left( \max\limits_{z,w \in \mathcal{K}_{xy}} \deg {\p}_{zw}(\varepsilon) \right)  \cdot  \Uplambda\left( \mathcal{N}_{xy}, 2 \, \mathcal{N}_{xy}^{\,^2}-2\, \mathcal{N}_{xy}+1 \right).
	\]
where the function $\Uplambda$ is an upper bound function for bounded convex polytopes given in~Theorem~\ref{thm:ub}. 
\end{enumerate}
\end{theorem}
%-----------------------------------------------
%-----------------------------------------------
\subsubsection*{\small \bf \textsf{Continuous-time Ollivier-Ricci curvature}}
The continuous time Ollivier-Ricci curvature $\olric(x,y)$ is defined (a la Ollivier) as the derivative
\[
\text{-} \, \nicefrac{d}{d\varepsilon}\restr_{\varepsilon=0} \; \nicefrac{\Was_1^{\dist}\left( \upmu_x^\varepsilon, \upmu_y^ \varepsilon \right)}{\dist(x,y)},
\]
whenever it exists \cite{Ol}. The original definition is intended for a continuous-time random walk generated by a Markov kernel, however the same definition can be used for more general $1$-parameter walks. We note that for the definition of continuous-time Ollivier-Ricci curvature, we just need to know the germ of the random walk at $\varepsilon=0$ however for simplicity we always assume that  $\varepsilon$ ranges from $0$ to $1$ but this is indeed of no significance in the theory. 
\par Let us point to an important class of continuous-time random walks that we call pleasant walks; these are walks of the form
	\[
\upmu_z^\varepsilon = \updelta_z + \varepsilon \upmu_z + \mathcal{R}^\varepsilon_z,
\]
whose $1$-jet is a time-analytic walk with some further second order asymptotic conditions on $ \mathcal{R}^\varepsilon_z$; see Definition~\ref{defn:pls-walk}. In particular $ \upmu_z$ and $ \mathcal{R}^\varepsilon_z$  are zero-mass signed measures. 
\par Also some important function spaces that frequently appear are the space of $1$-Lipschitz functions which is denoted by $\Lip(1)$ and the space of finite support functions which we denote by $\mathcal{C}_{\sf fs}(\G)$. 

\begin{theorem}
	The continuous-time Ollivier-Ricci curvature, $\olric(x,y)$, is well-defined for continuous-time  pleasant  walks. Furthermore, for pleasant walks with local $1$-jets, the following hold;
		\begin{enumerate}
	\item (limit-free formulation) $\olric$ is given by the limit-free variational formula
	\[
	\olric(x,y) = \inf_{\substack{f\in \Lip(1) \cap \mathcal{C}_{\sf fs} \\ \nabla_{xy} f=1}} \nabla_{yx} \mathcal{L} f,
	\]
	where the operator $\mathcal{L}$ is the initial velocity of the (not necessarily Feller) process corresponding to the given $1$-parameter random walk;
	\item (Lipschitz continuity 1) $\olric(x,y)$ is locally Lipschitz in $\upmu_z$ and in the distance $\dist$ (as variables);
		\item (Lipschitz continuity 2)  if $\upmu_{z}^\varepsilon$ as in the previous item is locally Lipschitz in $\omega$, then $\olric(x,y)$ is  locally Lipschitz continuous in the arguments $\omega$ and in the distance $\dist$ (as a variable);
		\item (minimizer) the limit-free formulation admits a minimizer;
		\item (locality) the minimizer can be localized to be supported in $\mathcal{K}_{xy}$;
\end{enumerate}
\end{theorem}
More details regarding the above theorem are to be found in~\textsection\thinspace\ref{sec:ct-olric}. 
\par We refer the reader to~\textsection\thinspace\ref{sec:oper-theo} for the precise definitions of operator-theoretic notions that we will shortly see. 
\begin{theorem}
	$\olric$ is well-defined for Markovian walks $e^{\varepsilon {\mathcal L}} \updelta_z$  when $\mathcal{L}$ is a good operator and it satisfies
	\begin{enumerate}
		\item $e^{\varepsilon \mathcal{L}}$  and $\text{-} \mathcal{L}$ satisfy rough comparison principles with range $0$;
		\item  $\left| \mathcal{L} d(z,\cdot)(z) \right| \le C(z)$. 
	\end{enumerate}
\end{theorem}
\par Good operators are translation invariant essentially self-adjoint operators that are roughly of divergence form; see~Definition~\ref{def:good-op}. We say $\G$ is locally $\dist$-finite if metric balls are finite.
\begin{theorem}
	Let $\G$ be locally $\dist$-finite. Then, $\olric$ is well-defined for Markovian walks $e^{\varepsilon {\mathcal L}} \updelta_z$  when $\mathcal{L}$ is a good operator and it satisfies
	\begin{enumerate}
	\item $e^{\varepsilon \mathcal{L}}$ satisfies rough comparison principle with range $0$;
	\item   $\mathcal{L}$ is a semi-local operator with range $R$;
	\item $\mathcal{L}$ satisfies rough comparison principle on $\Lip_1(\G)$ with range $2R$;
	\item $\left| \mathcal{L} \dist(z,\cdot)(w) \right| \le C(z)$ holds $\forall w \in \B_{2R}(z)$.
\end{enumerate}
\end{theorem}
We will refer the reader to~\textsection\thinspace\ref{sec:ex-heat} for the above existence  theorems in the spacial setting where the walk is the heat kernel i.e. where $\upmu_z^\varepsilon = e^{\varepsilon\Delta}\delta_z$.
%-----------------------------------------------
%-----------------------------------------------
\subsubsection*{\small \bf \textsf{Operator theoretic Ollivier-Ricci curvature}}
\par Let 
$
\mathcal{L}: \R^G \supset \mathsf{Dom}(\mathcal{L}) \to \R^G
$
be an operator with
$ \mathcal{C}_{\sf fs}(\G)  \subset  \mathsf{Dom}(\mathcal{L}) $.
We define operator-theoretic Ollivier-Ricci curvature by
\[
\olric_{\mathcal{L}}(x,y) := \inf_{\substack{ f \in \mathcal{C}_{\sf fs}(\G)\\ f \in \Lip(1)\\ \nabla_{xy} f = 1 }} \nabla_{yx}  \mathcal{L}f.
\]
\begin{theorem}
	Suppose $\mathcal{L}$ satisfies the following properties
	\begin{enumerate}
		\item $\dist(z,\cdot) \in \mathsf{Dom}\left( \mathcal{L} \right)$, $\forall z$;
		\item $\mathcal{L}$ is a rough differential operator;
		\item $\mathcal{L}$ satisfies a two-sided rough comparison principle.
	\end{enumerate}
	Then, $\olric_{\mathcal{L}}(x,y)$  is finite. 
\end{theorem}
See Definition~\ref{def:two-sided} for two-sided comparison principle. Roughly speaking, semi-local means $\mathcal{L}f(x)$ only depend on the values of $f$ on a metric ball around $x$; rough differential operator means the operator is translation invariant; see~\textsection\thinspace\ref{sec:op-theo-ric}. 
\begin{theorem}
	In locally $\dist$-finite $\G$, $\olric_{\mathcal{L}}(x,y)$ is finite for bounded semi-local rough differential operators $\mathcal{L}$. 
\end{theorem}
Here, boundedness of the operator is with respect to the sup-norm. Furthermore, notice in above theorems, for any two such operators $\mathcal{L}_1$ and $\mathcal{L}_2$, the concavity relation
	\[
	\olric_{t\mathcal{L}_1 + (1-t)\mathcal{L}_2} \ge t  \olric_{\mathcal{L}_1} + (1-t) \olric_{\mathcal{L}_2}
	\]
should clearly hold. 
%-----------------------------------------------
%-----------------------------------------------
\subsubsection*{\small \bf \textsf{Continuous-time generalized Ollivier-Ricci flows}}
The continuous-time generalized Ollivier-Ricci curvature flow equation ${\sf ORF}_{f,g,c}$ is  the ODE system
\[{\sf ORF}_{f,g,h} :
\begin{cases}\label{eq:ode}
	\dot{\omega} = f(t, \omega, \dist, \m, \olric )\\
	\dot{\dist} =  g\left(t, \omega, \dist, \m, \olric \right)\\
	\dot{\m} = h\left(t, \omega, \dist, \m, \oscal \right) 
\end{cases}
\]
for given functions $f,g$ and $h$.  See~\textsection\thinspace\ref{sec:CTRF} for more details.
\begin{theorem}
	Let $\G$ be a finite graph. Suppose a time-analytic random walk  $\upmu_x^\varepsilon$ depends on $\omega$ and $\m$ in a locally Lipschitz continuous way. Also suppose $f,g,h$ are each locally uniformly Lipschitz in the $t$ variable and locally Lipschitz in the other variables.  
	Then, for any initial data in the interior of te phase space, there exists a unique solution to this flow and the maximal solution exists as long as the primary edge weight $\omega$ stays non-negative and $m$ stay positive and $\dist$ stays a distance. In particular, the Ollivier-Ricci flow
	\[
	\dot{\omega} = -\olric \cdot \omega, \quad \dist= \dist_\omega.
	\]
	admits unique solutions. 
\end{theorem}
We note that once a local Lipschitz continuity of $\olric$ is at our disposal, we can show existence and uniqueness a wide range of similar ODE systems not necessarily restricted to the form \eqref{eq:ode}.
%-----------------------------------------------
%-----------------------------------------------
\subsubsection*{\small \bf \textsf{Miscellaneous}}
\par In order to get an upper bound on the number of polynomial pieces in discrete-time generalized Ollivier-Ricci curvature for time-polynomial local walks, we  needed to prove an upper bound theorem on the number of vertices  of a bounded convex polytope in $\R^N$ in terms of the number of its facet; the latter was achieved by using and slightly modifying the famous McMulen's upper bound theorem~\cite{McM}; see Theorem~\ref{thm:ub}. This might be also of independent interest to convex  geometers. 
\par Also we needed to discuss the Lipschitz sensitivity  of the optimal value function in linear programming problems, to perturbations of the constraints and objective functions; for that we built upon the perturbation theorem of Renegar~\cite{Ren}; this also might be of independent interest; see Theorem~\ref{thm:piecewise-anal}. 
%-----------------------------------------------
%-----------------------------------------------
% Section: Organization of materials  
%-----------------------------------------------
%-----------------------------------------------
\addtocontents{toc}{\protect\setcounter{tocdepth}{-1}}
\section*{\small \bf  Organization of the Materials}
\addtocontents{toc}{\protect\setcounter{tocdepth}{1}}
In~\textsection\thinspace\ref{sec:setup}, we present some calculus and differential geometric  tools on weighted graphs;~\textsection\thinspace\ref{sec:upper-bound}
is devoted to obtaining an upper bound on the number of vertices in a convex polytope when the number of facets and the ambient dimension is known;~\textsection\thinspace\ref{sec:sens-anal} discusses sensitivity analysis of the optimal value function in linear programming; in~\textsection\thinspace\ref{sec:oper-theo}, we gather some operator-theoretic notions and tools that we will need in later sections; 
the main \textsection\thinspace\ref{sec:edge-ric} discusses the definition, existence and properties of both discrete-time and continuous-time Ollivier-Ricci curvature are discussed using the results obtained in earlier sections and finally, in~\textsection\thinspace\ref{sec:CTRF}, we apply the results obtained to continuous-time generalized Ollivier-Ricci curvature flows to establish existence and uniqueness.
%-----------------------------------------------
%-----------------------------------------------
% Section: Acknowledgements
%-----------------------------------------------
%-----------------------------------------------
\addtocontents{toc}{\protect\setcounter{tocdepth}{-1}}
\section*{\small \bf  Acknowledgements}
\addtocontents{toc}{\protect\setcounter{tocdepth}{1}}
\vspace{-160pt}
	\begin{minipage}[t][7cm][b]{0.87\textwidth}
		\small
		\renewcommand{\labelitemi}{\raisebox{2pt}{\scalebox{1.2}{$\centerdot$}}}
		\begin{itemize}
	\item SL acknowledges partial support by the \emph{ Kazemi Ashtiani early career award}; awarded by the \emph{Iran's National Elites Foundation}; \medskip
		\item SL acknowledges partial support from \emph{IPM, Grant No. 1400460424} --  as part of the project: ``Geometric and analytic methods in studying complex systems''.
\end{itemize}
\end{minipage}
\normalsize
\vspace{10pt}
%-----------------------------------------------
%-----------------------------------------------
% Section: Setup and Notations
%-----------------------------------------------
%-----------------------------------------------
\section{A quick setup of discrete calculus}\label{sec:setup}
  %-----------------------------------------------
%-----------------------------------------------
%-----------------------------------------------
\subsection{Weighted graphs}
\par An un-directed locally finite weighted graph $\G$ is a non-negative symmetric primary weight function $\omega:\mathbb{Z}^{2} \to \R$ with vertex set
\[
V := \{x \in \Z  \;\; \text{\textbrokenbar}\;\;   \exists y \in \Z \;\; \text{s.t.} \;\;  \omega(x,y)>0 \;\; \text{or} \;\; \omega(x,y)>0   \},
\]
and edge set
\[
E := \{(x,y) \in \Z^{2} \;\; \text{\textbrokenbar}\;\;  \omega(x,y)>0    \}/\sim ,\quad (x,y)\sim (y,x).
\]
in which for every $x$, $\omega(x,\cdot)$ has finite support i.e. all vertices have finite combinatorial degrees. 
\par Throughout these notes, all graphs are assumed to be locally finite. For simplicity, $x\sim y$ means there is an edge between $x$ and $y$. We set $\omega_{xy} := \omega(x,y)$. The combinatorial degree of a vertex $x$ is denoted by $\deg_x$ while we signify the the weighted degree by
\[
\Deg_x :=  \nicefrac{1}{{\m}(x)} \sum_{y\sim x} \omega_{xy} =: {\substack{\omega\\ \m} \underset{y\sim x}{\sum}} 1;
\]
here, ${\substack{\omega\\ \m} \underset{y\sim x}{\sum}}$ is our adopted notation to represent the doubly (vertex and edge) weighted summation.
%-----------------------------------------------
%-----------------------------------------------
\subsubsection*{\small \bf  \textsf{Weighted path distance}}
There is a variety of ways to assign a distance function to a graph. We will only recall the most common one which is the weighted path distance. Let $\eta$ be a symmetric edge-weight function. The distance $\dist_\eta$ is defined by
\begin{align}\label{eq:dist}
\dist_\eta \left( x,z \right)= \inf_{z_0\sim z_1 \sim \dots \sim z_n} \sum_{i=0}^{n-1}\eta_{z_iz_{i+1}}.
\end{align}
For example, for $\eta \equiv 1$, one retrieves the combinatorial distance $\dist$.  Paths that minimize the  $\dist_\eta$ distance are called $\dist_\eta$-geodesics. 
%-----------------------------------------------
%-----------------------------------------------
%-----------------------------------------------
\subsection{Calculus tools}
Recall $\dist$ is a distance on $\G$ and we always assume $\left(  \G, \dist\right)$ is metrically complete. The differential-geometric objects defined in this section are mostly depending on the distance $\dist$.
%-----------------------------------------------
%-----------------------------------------------
\subsubsection*{\small \bf \textsf{Tangent space and space of directions}}
\begin{definition}[Direction space at a vertex]
The direction space at a vertex $x\in \G$ is given by
	\[
	\Dir_x(\G) := \left\{ \text{all the $\dist$-geodesics emanating from $x$} \right\}.
	\]
	So, $\mathsf{Geod}(\G) = \bigcup\limits_{x \in \G} \Dir_x(\G) =: \Dir \G$ is the set of all $\dist$-geodesics. We will denote the length of a geodesic $\upgamma$ by $\|\upgamma\|$. 
\end{definition}
Set
$
\mathbb{L}_x := \dist(x,\cdot)^{-1}(\G)
$
which is an ordered subset 
$
\left\{0 = \ell_{x,0} , \ell_{x,1}, \dots, \ell_{x,n_x}  \right\} \subset \R_{\ge 0},
$
providing all the distances realized by geodesics emanated from $x$;
here, $n_x$ is the combinatorial length the longest geodesic emanating from $x$. The set $\mathbb{L}_x$ can be written as the union of $\mathbb{L}_{x\sim y}$ (distances set off in the direction $x\sim y$) where
\[
\mathbb{L}_{x\sim y} := \left\{ \ell \in \mathbb{L}_x   \;\; \text{\textbrokenbar}\;\; \exists \upgamma \in \Dir_x(\G) \;\; \text{with}\;\; y \in \upgamma \;\; \text{and}\;\; \|\upgamma\| = \ell \right\}.
\]
\begin{definition}[Tangent space]
\par We define the tangent space at a vertex $x\in \G$ to be
	\[
	\Tan_x(\G) := \bigcup\limits_{y\sim x} \left\{  (x,y) \right\} \times \mathbb{L}_{x\sim y} / \sim.
	\]
where  $(x,y,0) \sim (x,z,0)$.  So a tangent vector specifies an ``infinitesimal direction'' $x\sim y$ and a length $\ell$.  Set $\Tan \G:= \bigcup\limits_{x \in \G} \Tan_x\G$. The length of tangent vectors are obviously defined by 
$
\| (x,y,\ell) \| = \ell.
$
\end{definition}
\par The tangent sphere, $\mathbb{S}^1_x \G$, is the set of all unit tangent vectors at $x$. The union of all these, is called the unit sphere bundle $\mathbb{S}^1 \G$; notice the latter might very well be an empty set. 
\par The tangent space $\Tan_x(\G)$ is manifestly depending on the distance $\dist$; a fact which is in contrast to the construction of tangent space in the smooth setting. 
%-----------------------------------------------
%-----------------------------------------------
\subsubsection*{\small \bf \textsf{Evaluation, exponential and logarithm maps}}
\par For any $\upgamma \in \Dir_x(\G)$, the distances realized by $\upgamma$ are
\[
\mathbb{L}_{x,\upgamma} := \dist(x,\cdot)^{-1}\left(\mathrm{image}(\upgamma)\right) = \left\{ 0 = \ell^0_{x,\upgamma}, \ell^1_{x,\upgamma}, \dots \ell^{n_{x,\upgamma}}_{x,\upgamma} \right\} \subset \mathbb{L}_x,
\]
where $n_{x,\upgamma}$ is the combinatorial length of $\upgamma$. Define the evaluation maps ${\sf e}: \mathrm{Geod}(\G) \times \mathbb{L} \to \G$ by
\[
{\sf e}\left( \upgamma, \ell^m_{\upgamma(0), \upgamma} \right) = \left\{ v \in \mathrm{image}(\upgamma)  \;\; \text{\textbrokenbar}\;\;  \dist(\upgamma_0, v) =\ell^m_{\upgamma(0), \upgamma}   \right\}, \quad m\le n_{\upgamma(0), \upgamma};
\]
So, $\mathsf{e}\left( \upgamma, \ell^m_{\upgamma(0), \upgamma} \right)$ simply gives the $m$-th vertex along the geodesic $\upgamma$. 
\begin{definition}[exponential and logarithm maps]
	The exponential map
	$
	\bar{\exp}_x: 	\Tan_x(\G) \to \G
	$
	is given by 
	\begin{align*}
	&\bar{\exp}_x(x,y,\ell) := \\ & \left\{ z \in \G \;\; \text{\textbrokenbar}\;\; \dist(x,z) =\ell \;\; \text{and there exists a geodesic from $x$ to $z$ passes through $y$}  \right\}.
	\end{align*}
	It is readily evident that $\exp_x$ is injective on $\B_{r_0(x)}(x)$ where $r_0(x) := \min_{y\sim x} \dist(x,y)$; hence, we dub $r_0$, $\mathrm{injrad}(x)$. In general, $\bar{\exp}_x$ and its inverse are multi-valued maps. The multi-valued inverse  which is denoted by $\bar{\log}_x: \G\to \Tan_x\G$ is defined via
	\[
\bar{\log}_x (z) := \left\{ e \in T\G\;\; \text{\textbrokenbar}\;\;  \bar{\exp}_x(e) \ni z   \right\}.
	\]
\end{definition}
\par There is a single-valued map
$
\mathrm{iv}_x: \Dir_x\G\to \Tan_x\G
$
given by 
\[
\mathrm{iv}_x \left(\upgamma \right) := \left( x, \upgamma_1, \mathrm{length}_\eta(\upgamma)\right),
\]
where $\upgamma:\mathbb{L}_x \to G$ is a combinatorial geodesic; ``$\rm{iv}$'' stands for initial velocity; the multi-valued inverse map
$
\mathrm{iv}^{-1}_x: \Tan_x\G\to \Dir_xG
$
which is given by
\[
\mathrm{iv}^{-1}_x \left( (x,y,\ell) \right) := \left\{\text{all geodesics emanating from $x$, passing through $y$ and having length $\ell$}  \right\},
\]
is the set of all points that are reachable with initial velocity $(x,y,\ell)$. Indeed, the map $\bar{\exp}$ factors through the map $\exp$ via the map $\mathrm{iv}^{-1}$; thus, one gets the single-valued map 
$
\exp_x: \Dir_x\G\to \G
$
and its multi-valued inverse $\log_x: \G\to \Dir_x\G$; see the diagram in Figure~\ref{Fig:1}. 
\begin{center}
	\begin{figure}
\begin{tikzpicture}
	\matrix (m) [matrix of math nodes, row sep=3em,
	column sep=3em]{
		 \Tan_x\G &&\Dir_x\G\\
		  \G& &  \\};
	\path[-stealth]
	(m-1-1) edge node[midway, left] {$ \bar{\exp}_x$} (m-2-1);
	\path[-stealth]
	(m-1-3) edge node[midway, below] {$\hspace{15pt} \exp_x$} (m-2-1);
		\path[-stealth]
	(m-1-1) edge node[above] {$\mathrm{iv^{-1}}$} (m-1-3);
\end{tikzpicture}
\caption{Commuting diagram (in the sense of multi-valued maps) of the exponential maps}
	\label{Fig:1}
\end{figure}
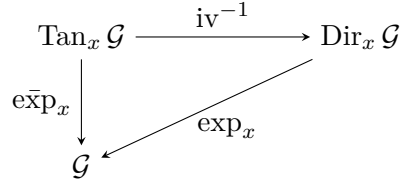
\end{center}
%-----------------------------------------------
%-----------------------------------------------
\subsubsection*{\small \bf  \textsf{Differentiation operators}}
\begin{definition}[Differentiation operators]
	The tangent map of a function $f$ is the multi-valued map
$
\nabla_{\centerdot} \; f: \Tan \G\to \R
$
defined by
\[
\bar{\nabla}_e f := \ell^{-1}\left( f(\exp_x(e)) - f(x) \right), \quad (x,y, \ell) = e \in \Tan_x \G,
\]
where, for $n=0$, we set it equal to $0$. The tangent map factors through a single-valued map with domain $\mathsf{Geod}(\G) = \Dir \G$  which we denote by $\nabla$ and call the directional derivative. So with a slight abuse of notations and terminology, for $x$ and $y$ in $\G$, we  write
\[
\nabla_{xy} f := \nabla_{\log_x{y}} f = \dist(x,y)^{-1} \left( f(y) - f(x) \right),
\]
which is the derivative of $f$ in the multi-valued direction $\log_x{y}$. Suppose $\upgamma$ is one such direction, then we sometimes also write $\nabla_{\upgamma}f$ for brevity.   
\end{definition}
\par So with the definition of directional derivative, a function $f$ being $\sf L$-Lipschitz is equivalent to $|\nabla f| \le {\sf L}$; we denote the lowest such ${\sf L}$ by $\lip(f)$. 
\par A special case is when $\dist$ is the combinatorial distance; in this case, we will also denote by $\nabla^1_{\bullet}u: \mathbb{S}^1\G\to \R$, the map that is given by
\[
\nabla^1_{xy} u := \bar{\nabla}_{(x,y,1)}f = f(y) - f(x).
\]
%-----------------------------------------------
%-----------------------------------------------
\subsubsection*{\small \bf \textsf{Discrete Laplacian(s)}}
\par The most general form of discrete Laplacian is given by
\[
\Delta f (x):= \m(x)^{-1} \sum_{x \sim y} \left( f(y) - f(x)   \right) \omega_{xy} =: {\substack{\omega\\ \m} \underset{y\sim x}{\sum}}\left( f(y) - f(x)   \right).
\]
\par The special case of the combinatorial graph Laplacian $\mathbf{\Delta}$ is obtained when $\m \equiv 1$ and $\omega \equiv 1$, i.e.
\[
\Delta_{\sf c} f (x):= \sum_{x \sim y} \left( f(y) - f(x)   \right).
\]
\par The normalized Laplacian is also a common one and is obtained if we set ${\m}(x) = \omega (x) :=  \sum_{y\sim x} \omega_{xy}$ i.e. $\Deg \equiv 1$;
\[
\Delta_{\sf n} f (x):= \omega(x)^{-1}\sum_{x \sim y} \left( f(y) - f(x)   \right)  \omega_{xy}.
\]
The normalized Laplacian is particularly important since it is the infinitesimal generator of the weighted random walk; e.g. see~\cite{Barlow}.  
%-----------------------------------------------
%-----------------------------------------------
% Section: Setup and Notations
%-----------------------------------------------
%-----------------------------------------------
\section{An upper bound result for convex polytopes}\label{sec:upper-bound}
In the next section, we will be dealing with linear programming problems with feasible sets that are bounded convex polytopes in some $\R^N$. One estimate that will be important to us is an upper bound for the number of vertices of a convex polytope in $\R^N$ that has $m$ facets (faces of co-dimension $1$). 
%-----------------------------------------------
%-----------------------------------------------
\subsection{\textbf{Upper bound}}
\par We will need the following two elementary lemmas.
\begin{lemma}\label{lem:poly-1}
	Set 
	$
	\h_1(m,r) := \binom{m-r}{r}, 2r \le m;
	$
	the following hold.
	\begin{enumerate}
		\item For fixed $m$,  $\h_1(m,r)$ is an increasing function of $r$ for $0\le r \le  \lfloor  r_{11}(m)\rfloor $ and is  a decreasing  function of $r$  for $  r  \ge \lceil r_{11}(m) \rceil $ where 
		\[
		r_{11}(m) = \nicefrac{\left(5m-3 -\left( 5m^2 + 10m + 9 \right)^{\nicefrac{1}{2}}\right)}{10};
		\]
		\item as a result of (1), one has $\h_1(m,r) \le \h_1\left( m, \lfloor  r_{11}(m) \rfloor \right)  \vee  \h_1\left( m, \lceil r_{11}(m) \rceil  \right)     $.
	\end{enumerate}
\end{lemma}
\begin{proof}[\footnotesize \textbf{Proof}]
\par Set
	\begin{align*}
	\h_1(m,r+1)\h_1(m,r)^{-1} &= \binom{m-r-1}{r+1}	\binom{m-r}{r}^{-1} \\ & = \nicefrac{(m-2r)(m-2r-1)}{(m-r)(r+1)} := \mathcal{R}_1(m,r).
	\end{align*}
Treating $\mathcal{R}_1$ as a function of real variables, $\mathcal{R}_1(m,r) \ge 1$ is equivalent to 
	\[
	q_1(r) = 5r^2 + \left( -5m+3 \right)r + m^2-2m \ge 0.
	\]
	$q_1(r)$ has real roots $r_1(m) < r_2(m)$ given by
	\[
	r_{11}(m) = \nicefrac{\left( 5m-3 -\left( 5m^2 + 10m + 9\right)^{\nicefrac{1}{2}}\right)}{10}, \quad   r_{12} = \nicefrac{\left( 5m-3 +\left( 5m^2 + 10m + 9\right)^{\nicefrac{1}{2}}\right)}{10}.
	\]
\par It is easy to verify that	for $m\ge 2$, both roots are non-negative. This means for fixed $m\ge 2$, $\mathcal{R}(m,r) \ge 1$  holds for  $0\le r \le r_1(m) $ and $r > r_2(m)$,  and $\mathcal{R}(m,r)  \le 1$ for $r_1 (m) \le r \le r_2(m)$; this verifies (1). 
\par We also have
	\[
	\nicefrac{m}{2} <  \nicefrac{\left( 5m-3 + \left( 5m^2 + 10m  + 9\right)^{\nicefrac{1}{2}}\right)}{10} = r_{12}(m).
	\]
So, as $r$ increases to $\lfloor r_1(m) \rfloor$, $h(m,r)$ increases; and as $r$ increases from $\lceil  r_1(m)  \rceil $ all the way up to $\lfloor  \nicefrac{m}{2} \rfloor$, $\h_1(m,r)$  will decrease. Therefore, its maximum must be attained at either $\lfloor r_1(m) \rfloor$ or at $\lceil  r_1(m)  \rceil $ ;  this proves (2). 
\end{proof}
\begin{lemma}\label{lem:poly-2}
Set 
		$
		\h_2(m,r) := \binom{m-r}{r-1}, 2r \le m + 1;
		$
then, the items in Lemma~\ref{lem:poly-1} hold, using $r_{21}(m)$ instead where
		\[
	r_{21}(m) = \nicefrac{\left( 5m+2 - \left( 5m^2 + 4 \right)^{\nicefrac{1}{2}}\right)}{10}.
	\]
\end{lemma}
\begin{proof}[\footnotesize \textbf{Proof}]
Similar to the proof of Lemma~\ref{lem:poly-1}, we proceed as follows.
	\begin{align*}
	\h_2(m,r+1)\h_2(m,r)^{-1} &= \binom{m-r-2}{r}	\binom{m-r-1}{r-1}^{-1} \\ &= \nicefrac{(m-2r)(m-2r+1)}{(m-r)(r)} := \mathcal{R}_2(m,r).
	\end{align*}
	\[
	\mathcal{R}_2(m,r)\ge 1   \iff  	q_2(r) = 5r^2  + (-5m -2)r +  m^2 + m \ge 0.
	\]
\par Roots of $q_2$ are
\[
r_{21} = \nicefrac{\left( 5m+2 - \left( 5m^2 + 4\right)^{\nicefrac{1}{2}}\right)}{10}, \quad r_{22} = \nicefrac{\left( 5m+2 + \left( 5m^2 + 4\right)^{\nicefrac{1}{2}}\right)}{10};
\]	
and both are positive. Furthermore,
 $
\nicefrac{(m+1)}{2} < r_{22}.
$
clearly holds; so we can argue in the same fashion as in Lemma~\ref{lem:poly-1}. 
\end{proof}
\par Let us recall that McMullen's upper bound theorem states that among all $d$-dimensional polytopes with $n$ vertices, the cyclic ones maximize the number of $i$-dimensional faces for $i=1, \cdots, d-1$.  We refer the reader to the exposition~\cite{HRZ} and the original paper \cite{McM} for more details. 
\par The following proposition is a direct consequence of the upper bound theorem for convex polytopes in combination with the concept of dual polytopes; see~\cite{McM,HRZ,Seid}. The number of facets in nontrivial cases satisfy $m\ge 3$ so the above lemmas apply. 
\begin{proposition}
	Suppose $\mathfrak{P}$ is a $d$-polytope with $m$ number of facets and $n$ number of vertices. Then
	$
	n \le {\rm F}_{\sf cyc}(d, m),
	$
	where ${\rm F}_{\sf cyc}$ is given by
	\[
	{\rm F}_{\sf cyc}(d,m) := \binom{m - \lceil \nicefrac{d}{2}  \rceil }{\lfloor  \nicefrac{d}{2} \rfloor} + \binom{m - 1 - \lceil \nicefrac{(d-1)}{2}  \rceil }{\lfloor  \nicefrac{(d-1)}{2} \rfloor},
	\]
	which is the number of facets in a  cyclic $d$-polytope with $m$ vertices.
\end{proposition}
\begin{proof}[\footnotesize \textbf{Proof}]
First notice that going to the dual polytope, we deduce $d< m$ so ${\rm F}_{\sf cyc}(d,m)$ is well-defined. Now suppose 
$
	n > {\rm F}_{\sf cyc}(d, m) 
$
	then, the dual polytope $\mathfrak{P}^*$ has more than ${\rm F}_{\sf cyc}(d, m)$ facets; hence, by McMullen's upper bound theorem, the dual polytope $\mathfrak{P}^*$ must have more than $m$ vertices. This means $\mathfrak{P}$ must have more than $m$ facets which is a contradiction. 
\end{proof}
\begin{remark}\label{rem:other-form}
	\phantom{}\hspace{-1pt}\textit{Notice that the function $ {\rm F}_{\sf cyc}(d , m) $ can be written as
	\begin{align*}
	 {\rm F}_{\sf cyc}(d , m)  &= \begin{cases}  \binom{m-r}{r} + \binom{m-r-1}{r-1}& \text{if}\quad  d = 2r   \\  2 \binom{m-r-1}{r} &  \text{if}\quad  d = 2r+1    \end{cases} \\ &=  \begin{cases}  \h_1(m,r) + \h_2(m-1,r) & \text{if}\quad   d = 2r   \\  2 \h_1(m-1,r) & \text{if}\quad   d = 2r+1    \end{cases}.
	\end{align*}
}	
\phantom{}
\end{remark}
\begin{theorem}\label{thm:ub}
	Suppose $\mathfrak{P}$ is a bounded convex polytope in $\R^N$ with $m$ number of facets and $n$ vertices. Then
	\begin{align*}%\label{eq:main-bound}
n\le 2 \h_{111}(m-1) \vee \Big( \h_{111}(m)  + \h_{221}(m-1) \Big) := \Uplambda(N,m),
	\end{align*}
with the convention
\begin{align}\label{eq:h-ijk}
\h_{ijk}(m) := \h_i\left( m, \lfloor  r_{jk}(m)  \rfloor \wedge \lfloor  \nicefrac{N}{2} \rfloor   \right) \vee \h_i\left( m, \lceil  r_{jk}(m) \rceil  \wedge \lfloor  \nicefrac{N}{2} \rfloor    \right).
\end{align}
\end{theorem}
\begin{proof}[\footnotesize \textbf{Proof}]
	Based on Lemmas~\ref{lem:poly-1} and~\ref{lem:poly-2}, we clearly have
	\[
	2 \h_1(m-1,r)  \le 2\h_{111}(m-1),
	\]
	and
	\[
	\h_1(m,r) + \h_2(m-1,r)  \le \h_{111}(m)  + \h_{221}(m-1),
	\]
	which by virtue of Remark~\ref{rem:other-form} and $r\le \lfloor \nicefrac{d}{2} \rfloor \le  \lfloor \nicefrac{N}{2} \rfloor$, gives the desired conclusion. 
\end{proof}
\begin{theorem}[sharpness]
	The upper bound $\Uplambda$ is achieved in some cases. So by definition, it is a sharp bound. 
\end{theorem}
\begin{proof}[\footnotesize \textbf{Proof}]
 \par Set $N=3$ and $m$ large, we get the same bound but this time, by Remark~\ref{rem:other-form}, we know $2m-4$ is sharp  for  cyclic 3-polytopes  in $\R^3$ with $m$ facets.  Since the upper bound $\Uplambda$ is achieved in some cases, it is a sharp bound. 
\end{proof}
\begin{example}\phantom\hspace{-3pt}(sharpness and non-sharpness)
	One can naturally find cases in which the upper bound $\Uplambda$ is an overkill.  Set $N=2$ and consider a regular polytope in $\R^2$ with  $m>4$ vertices and facets. for $m$ sufficiently large  Then
	\[
2 \h_{111}(m-1) = 2 \h_{1}(m-1,1) = 2(m-2) = 2m-4,
    \]
and
\[
 \h_{111}(m)  + \h_{221}(m-1) = \h_{1}(m,1) + \h_2(m-1,1) = m,
\]
so
	\[
	\Uplambda(m,2) = 2m-4 \vee m = \begin{cases} m & m\le4 \\ 2m-4 & m> 4   \end{cases},
	\]
which shows the upper bound is not sharp in this cases if $m>4$ and is only sharp for triangles. 
\end{example}
\begin{corollary}
	A crude upper bound is given by
		\begin{align*}
		&\Uplambda(N,m) \le\\
		&\quad	\Big( \max\limits_{ \lfloor \nicefrac{m}{5}  \rfloor -1  \le r\le  \lceil \nicefrac{3m}{10} \rceil  \wedge \lfloor   \nicefrac{N}{2}\rfloor }   2 \h_1\left( m-1, r \right)  \Big) \\ & \qquad \vee   \left( \max\limits_{\lfloor \nicefrac{m}{5}  \rfloor  -1 \le r\le  \lceil \nicefrac{3m}{10} \rceil  \wedge \lfloor   \nicefrac{N}{2}\rfloor} \h_1(m,r) +  \max\limits_{\lfloor \nicefrac{m}{5}  \rfloor  \le r\le  \lceil \nicefrac{3m}{10} \rceil  \wedge \lfloor   \nicefrac{N}{2}\rfloor } \h_2(m-1,r)  \right).
	\end{align*}
\end{corollary}
\begin{proof}[\footnotesize \textbf{Proof}]
	From the inequalities
	\[
	2(m+1)\le 	\left( 5m^2 + 10m + 9 \right)^{\nicefrac{1}{2}} \le 3(m+1),
	\]
	we deduce
	\[
	\nicefrac{m}{5} - \nicefrac{3}{5}  \le   r_{11}(m) \le \nicefrac{3m}{10} - \nicefrac{1}{2},
	\]
and
	\[
	\nicefrac{m}{5} - \nicefrac{4}{5}  \le   r_{11}(m-1) \le \nicefrac{3m}{10} - \nicefrac{4}{5}.
	\]
\par Similarly by
	\[
	2m\le 	\left( 5m^2 + 4\right)^{\nicefrac{1}{2}} \le 3m,
	\]
we get
	\[
	\nicefrac{m}{5} + \nicefrac{1}{5}    \le r_{21}(m) \le \nicefrac{3m}{10} - \nicefrac{1}{5},
	\]
and
\[
\nicefrac{m}{5}   \le r_{21}(m-1) \le \nicefrac{3m}{10} - \nicefrac{1}{2}.
\]
As a result,
\[
\lfloor  	\nicefrac{m}{5}  \rfloor  -1   \le  \lfloor  r_{11}(m)  \rfloor \le \lceil  r_{11}(m)  \rceil \le \lceil  \nicefrac{3m}{10}   \rceil,
\]
as well as
\[
\lfloor  	\nicefrac{m}{5}  \rfloor  -1   \le  \lfloor  r_{11}(m-1)  \rfloor \le \lceil  r_{11}(m-1)  \rceil \le \lceil  \nicefrac{3m}{10}   \rceil.
\]	
\par Similarly we have
\[
\lfloor  	\nicefrac{m}{5}  \rfloor   \le  \lfloor  r_{21}(m)  \rfloor \le \lceil  r_{21}(m)  \rceil \le \lceil  \nicefrac{3m}{10}   \rceil,
\]
and
\[
\lfloor  	\nicefrac{m}{5}  \rfloor    \le  \lfloor  r_{21}(m-1)  \rfloor \le \lceil  r_{21}(m-1)  \rceil \le \lceil  \nicefrac{3m}{10}   \rceil.
\]	
\par Therefore, using Theorem~\ref{thm:ub}, we get the conclusion. 
\end{proof}
\begin{theorem}\label{thm:dependence}
For constant $m$, $\Uplambda(N,m)$ is constant for $N>m$.
\end{theorem}
\begin{proof}
	When $N>m$, it follows
	\[
	\lfloor \nicefrac{N}{2} \rfloor \ge \lfloor \nicefrac{N}{2} \rfloor > r_{i1}(m) > r_{i1}(m-1), \quad i=1,2,
	\]
so the dependence on $N$  in the RHS of  \eqref{eq:h-ijk} disappears. 
\end{proof}
%-----------------------------------------------
%-----------------------------------------------
\subsection{Asymptotics}
The Theorem~\ref{thm:ub} immediately gives the following asymptotics.
\begin{theorem}
As $m,N\to \infty$ with $N>m$, we get
\[
\Uplambda(N,m) \approx 2 \binom{\lfloor  0.73m \rfloor }{\lfloor 0.27m \rfloor}.
\]
This asymptotic bound is majorized by $\left(2.70e\right)^{\nicefrac{m}{3.70}}$ and is minorized by $2.70^{\nicefrac{m}{3.70}} $. 
\end{theorem}
\begin{proof}
Keeping $m$ fixed, and for sufficiently large $N > 2m$, from Theorem~\ref{thm:dependence}, we get
\[
\Uplambda(N,m) = 2 \bar{\h}_{111}(m-1) \vee \Big( \bar{\h}_{111}(m)  + \bar{\h}_{221}(m-1) \Big),
\]
where
\[
\bar{\h}_{ijk}(m) := \h_i\left( m, \lfloor  r_{jk}(m)  \rfloor  \right) \vee \h_i\left( m, \lceil  r_{jk}(m) \rceil   \right).
\]
\par So as $m,N \to \infty$ with $N> 2m$, we get
\[
r_{11}(m), r_{21}(m) \approx \nicefrac{1}{2}\left( 1- \nicefrac{1}{5^{\nicefrac{1}{2}}}   \right)m \approx 0.27m;
\]
and as a result,
\[
\Uplambda(N,m) \approx \binom{\lfloor  0.73m \rfloor }{\lfloor 0.27m \rfloor}.
\] 
Straightforward combinatorial calculations show
\begin{align}\label{eq:asymp}
2.70^{\nicefrac{m}{3.70}}  &\approx \left( \nicefrac{\lfloor  0.73m \rfloor }{\lfloor 0.27m \rfloor}  \right)^{\lfloor 0.27m \rfloor } \notag \\  &  \lessapprox  \Uplambda(N,m) \\ & \lessapprox  e^{\lfloor 0.27m \rfloor } \left( \nicefrac{\lfloor  0.73m \rfloor }{\lfloor 0.27m \rfloor}  \right)^{\lfloor 0.27m \rfloor }   \notag \\ &\approx \left(2.70e\right)^{\nicefrac{m}{3.70}};\notag 
\end{align}
in which, the inequalities 
\[
\left(  \nicefrac{n}{k} \right)^k \le \binom{n}{k} \le \left(  \nicefrac{en}{k} \right)^k,
\]
are applied. 
\end{proof}
\begin{remark}\phantom{}\hspace{-1pt}\textit{In this special case,  \eqref{eq:asymp}  provides a better bound than the asymptotic bound $\mathcal{O}(m^{\nicefrac{N}{2}})$ that follows from~\cite{Seid} since
	\[
	m^{\nicefrac{N}{2}} = e^{\nicefrac{(\ln m) N}{2}} \gtrapprox e^{\nicefrac{m\ln m}{2}}.
	\]}
	\phantom{}
\end{remark}
%-----------------------------------------------
%-----------------------------------------------
% Section: Edge Ricci curvatures
%-----------------------------------------------
%-----------------------------------------------
\section{Sensitivity analysis for LP problems}\label{sec:sens-anal}
%-----------------------------------------------
%-----------------------------------------------
%-----------------------------------------------
\par Consider the linear programming problem 
\[
{\sf LP}_{\bf d} :=	\begin{cases} \min {\bf c} \cdot \hat{x} \\ {\bf a}\hat{x} \le {\bf b} \\ \hat{x} \ge 0 \end{cases}, \quad {\bf d}:= ({\bf a}, {\bf b}, {\bf c}),
\]
in $\R^n$ in which  ${\bf c}\in \R^n$, ${\bf b} \in \R^m$ and ${\bf a}_{m\times n}$ is a matrix; to this problem, we assign the vector   ${\bf d}:=({\bf a}, {\bf b} , {\bf c})$ in $\R^{mn+m+n}$.  Recall that the dual problem is given by
\[
{\sf LP}^*_{\bf a,b,c} :=	\begin{cases} \max \hat{y}\cdot {\bf b}\\ y{\bf a} \ge {\bf c} \\ \hat{y} \ge 0 \end{cases}.
\]
\par Let  ${\mathcal{ PF}}$ denote the set of all primary feasible ${\bf d}$ and ${\mathcal{PI}}$, the set of all primary infeasible ${\bf d}$; also let ${\mathcal{DI}}$ be the set of all ${\bf d}$ such that the dual problem is infeasible. 
\par For any such feasible vector ${\bf d } \in \mathcal{ PF}$, denote the optimal value by $\mathrm{val}({\bf d})$.  Due to finite dimensionality in our case,  the following theorem holds with any chosen norm on $\R^{mn+m+n}$.
\begin{proposition}[Lipschitz continuity of the optimal value~\cite{Ren}]\label{prop:lip-lp}
	Consider the linear programming problem ${\sf LP}_{\bf d_0}$ with $\uprho_{\bf d_0}:= \min\{ \dist\left( {\bf d},  \mathcal{PI} \right),   \dist\left( {\bf d},  \mathcal{DI}  \right)  \}>  0$.   There exists $L\left( \uprho_{\bf d_0},  {\bf d_0} \right) >  0$ such that for any ${\bf d} \in \B_{\uprho_{\bf d_0}} \left( {\bf d_0} \right)$, the local Lipschitz property
	\[
	\left| \mathrm{val}({\bf d}) - \mathrm{val}\left( {\bf d_0}\right)   \right| \le  L\left( \uprho_{\bf d_0},  {\bf d_0} \right) \left \|  {\bf d}  -  {\bf d_0}  \right\|, 
	\]
	holds true, with the convention $\infty-\infty = 0$.  i.e. $\mathrm{val}({\bf d})$ is  locally Lipschitz in the interior of  ${\sf PF}$ (with induced topology; in our case, this coincides with the so-called relative interior). We will denote this interior by $\mathrm{relint}({\sf PF})$.
\end{proposition}
%-----------------------------------------------
%-----------------------------------------------
%-----------------------------------------------
\subsection{Perturbation of the objective function}\label{sec:perturb-obj}
\par We wish to focus on the perturbation of the objective function. Let us start with the following elementary lemma. 
\begin{lemma}\label{lem:perturb-1}
	Let ${\bf c}_\varepsilon$, $\varepsilon \in [0,1]$ be a continuous one parameter vector. Suppose a set of  $\varepsilon$-independent constraints are given that give rise to a constant feasible set that is a bounded convex polytope $\mathfrak{Q}$; also assume the linear programming problem $\min {\bf c}_\varepsilon \cdot \hat{x}$ with the said constraints is well-posed for all $\varepsilon$. 
	\par Then, for each $\varepsilon$,  the optimal values can only be attained at the vertices; also the vertices that achieve the optimal value are included in a facet of ${\mathfrak Q}$. For any two such vertices $\hat{v}_\varepsilon$ and $\hat{w}_\varepsilon$, we have ${\bf c}_\varepsilon \cdot \left(  \hat{v}_\varepsilon -  \hat{w}_\varepsilon\right) = 0$. 
\end{lemma}
\begin{proof}[\footnotesize \textbf{Proof}]
	The fact that optimal values are attained at vertices is a standard fact of linear programming; indeed, no interior point can achieve the optimal value since otherwise subtracting a small positive multiple of ${\bf c}_{\varepsilon}$ from the interior point provides a point in ${\mathfrak Q}$ which violates the optimality. Suppose two vertices achieve the optimal value, then the optimal value will also be achieved along the line segment that joins these vertices; hence, the line segment must be included in a facet of the convex polytope ${\mathfrak Q}$. The perpendicularity claim should be clear. 
\end{proof}
%-----------------------------------------------
%-----------------------------------------------
\subsubsection*{\small \bf \textit{\textsf{Notation}}}
In the setting of Lemma~\ref{lem:perturb-1}, we denote by ${\mathcal{F}^{op}}\left( \varepsilon \right)$ the set of vertices of ${\mathfrak Q}$ that  achieve the optimal value (they are sometimes called \emph{basic feasible solutions}). We will denote the optimal value by $\mathrm{val}(\varepsilon)$. We denote the set of all vertices of ${\mathfrak Q}$ by $\mathcal{V}$. 
\begin{definition}[switching time]\label{defn:switch}
	An $\varepsilon_0$ will be called a switching time if there exists $\varepsilon_i \to \varepsilon$ such that $ \mathcal{F}^{op} \left( \varepsilon_i \right) \cap \mathcal{F}^{op} \left( \varepsilon_0\right) = \diameter$. We let $\mathcal{S}$ denote the set of switching times. 
\end{definition}
\begin{lemma}\label{lem:perturb-2}
	Suppose ${\bf d}_{\varepsilon} = \left( {\bf a}, {\bf b}, {\bf c}_{\varepsilon} \right)  \in \mathrm{relint}  \left(\mathcal{PF}\right)\cap   \mathrm{relint}\left(    \mathcal{PF} \right)$. Then, at a each switching time $\varepsilon_0$, $\left| {\mathcal{F}^{op}}\left( \varepsilon_0 \right) \right|\ge2$ holds. 
\end{lemma}
\begin{proof}[\footnotesize \textbf{Proof}]
	Obviously, since $\mathcal{V}$ is finite, there exists an $r_\varepsilon>0$, such that for $\hat{v} \in \mathcal{V} \smallsetminus {\mathcal{F}^{op}}(\varepsilon)$,
	\[
	\left| {\bf c}_\varepsilon \cdot \left( \hat{v} - \hat{w} \right) \right| > r_\varepsilon, \quad \forall \hat{w} \in {\mathcal{F}^{op}}(\varepsilon).
	\]
\par 	Now suppose contrary to the claim,  $ {\mathcal{F}^{op}}\left( \varepsilon_0 \right) = \{ \hat{v}_0 \}$ is a singleton.  Enumerate the vertices of ${\mathfrak Q}$ as $\mathcal{V} = \left\{ {\hat v}_0, {\hat v}_1, \cdots, {\hat v}_l   \right\} $.  
	Consider the functions 
	\begin{align}\label{eq:h-i}
	h_{i}(\varepsilon) := {\bf c}_{\varepsilon} \cdot {\hat v}_i.
	\end{align}
	Since $\varepsilon_0$ is a switching time, for the sequence $\varepsilon_k \to \varepsilon$ as in Definition~\ref{defn:switch}, it follows
	\begin{align}\label{eq:lip-2}
	\left| h_{i}\left( \varepsilon_k \right) - h_{0}\left( \varepsilon_0 \right) \right| >  r_\varepsilon, \quad \forall i\ge 1. 
	\end{align}
\par For any $k$, let $\mathrm{val}\left( \varepsilon_k \right)  = {\bf c}_{\varepsilon_k}  \cdot \hat{v}_{i_k} = h_{i_k} \left( \varepsilon_k \right)$; then, we must have $1\le i_k \le l$.  From \eqref{eq:lip-2}, we deduce
\[
\left| \mathrm{val}(\varepsilon_k)  -   \mathrm{val}(\varepsilon_0)   \right| >  r_{\varepsilon_0};
\] 
letting $k \to \infty$, this contradicts the Lipschitz continuity established by Proposition~\ref{prop:lip-lp}. 
\end{proof}
\begin{definition}[trivial perturbation]
	We say a perturbation ${\bf c}_{\varepsilon}$ is trivial if for all $\varepsilon$, $\mathcal{F}^{op}(\varepsilon) = \mathcal{V}$ holds i.e.  when for all times, all vertices achieve the optimal value. 
\end{definition}
\par For the rest of this section, we assume the perturbation considered is non-trivial. 
\begin{lemma}\label{lem:finite-switch}
	Suppose ${\bf c}_\varepsilon$, $\varepsilon \in [0,1]$ satisfies the hypotheses of the Lemmas~\ref{lem:perturb-1} and \ref{lem:perturb-2} and in addition, is analytic in $\varepsilon$ on the interval $(0,1)$; then for any $0<\updelta<1$, ${\mathcal S} \cap [\updelta, 1-\updelta]$ is a finite set. 
\end{lemma}
\begin{proof}[\footnotesize \textbf{Proof}]
Let $\mathcal{V} = \left\{ \hat{v}_0, \cdots, \hat{v}_l   \right\} $ be an enumeration of vertices of ${\mathfrak Q}$. Consider the analytic functions $h_i$ given by \eqref{eq:h-i}; by the hypotheses, $h_i$ are analytic in $\varepsilon$. 
	\par For every $i \neq j$, by analyticity, it follows either
	\[
	\left\{\varepsilon \;\; \text{\textbrokenbar}\;\;   h_{i}(\varepsilon)  = h_{j}(\varepsilon)  \right\}  \cap [\updelta, 1-\updelta],
	\]
	is a discrete set or $h_i \equiv h_j$ on $(0,1)$. Let $\varepsilon_0$ be a switching time and let $\varepsilon_k$ be a sequence as in Definition~\ref{defn:switch};
	without loss of generality, we assume $\hat{v}_0 \in {\mathcal{F}^{op}}\left( \varepsilon_0 \right)$. Take a sequence $\hat{v}_{i_{k}} \in {\mathcal{F}^{op}}\left( \varepsilon_k \right)$. So there exists $k_0$ and a sequence of natural numbers $m$ such that $\hat{v}_{i_{m}} = \hat{v}_{i_{k_0}} $ for all values of $m$; namely, we have extracted a constant subsequence of vertices. So by (Lipschitz) continuity of optimal values, one gets
	\[
	\mathrm{val}\left( \varepsilon_m \right) = h_{i_{k_0}} \left( \varepsilon_m \right)  \to  h_{i_{k_0}} \left( \varepsilon_0\right)  =  \mathrm{val}\left( \varepsilon_0) \right) = h_0 \left( \varepsilon_0 \right).
	\]
	which means $v_{i_{k_0}} \in {\mathcal{F}^{op}}\left( \varepsilon_0 \right)$ and furthermore, we know $h_{i_{k_0}} \not\equiv h_{0}$. So at every switching time, we find two distinct analytic functions $h_{i_{k_0}}$ and $h_{0}$ that take the same values. 
	\par Let
	$
	\mathcal{H} := \left\{ h_i \right\}
	$
	be the set of all distinct $h_i$; then if $\mathcal{H}$ has only one element, every $\varepsilon$ is a switching time and the optimal value is attained at all vertices at all times which is not possible \emph{since the perturbation is non-trivial}. If $|\mathcal{H}| \neq 1$, by pigeonhole principle, we must have
	\[
 \mathcal{A}:=	\left\{\varepsilon  \;\; \text{\textbrokenbar}\;\;   h_{i}(\varepsilon)  = h_{j}(\varepsilon) \; \text{for $h_i \not\equiv h_j \in \mathcal{H}$} \right\}  \cap [\updelta, 1-\updelta],
	\]
	is also a finite set. Since $\mathcal{S} \subset \mathcal{A}$, the conclusion follows. 
\end{proof}
\begin{theorem}\label{thm:piecewise-anal}
	In the setting of the previous Lemmas, $\mathrm{val}(\varepsilon)$ is a piece-wise analytic and locally Lipschitz function on $(0,1)$. 
\end{theorem}
\begin{proof}[\footnotesize \textbf{Proof}] We only need to argue piece-wise analyticity. By Lemma~\ref{lem:finite-switch}, we know there are finitely many switching times in $[\updelta, 1-\updelta]$ for $\updelta>0$ small. Between any two such switching times, the optimal value coincides with one of the analytic functions $h_i$; hence, the piece-wise analyticity holds within $[\updelta, 1-\updelta]$.  Since $0<\updelta<1$ can be arbitrary small, we deuce the optimal value is also piece-wise analytic on $(0,1)$; however, notice that \emph{there might be an infinitely many number of pieces (accumulating at the endpoints $0$ and/or $1$)}. 
\end{proof}
\subsection{Finite number of pieces}
A sufficient condition for ensuring a finite number of analytic pieces for the optimal value function is given in the following theorem. 
\begin{theorem}\label{thm:anal-ext}
	Assume the hypotheses of the previous theorem. In addition, suppose for some small $\updelta>0$, ${\bf c}_\varepsilon$ admits an analytic continuation over $(0-\updelta, 1+\updelta)$; suppose ${\bf c}_{\varepsilon}$ satisfies the hypotheses of Theorem~\ref{thm:piecewise-anal}, then  $\mathrm{val}(\varepsilon)$ has finitely many pieces in $[0,1]$. 
\end{theorem}
\begin{proof}[\footnotesize \textbf{Proof}] 
	We know $\mathrm{val}(\varepsilon)$  has finitely many analytic pieces within any $[\updelta, 1-\updelta]$.  Since ${\bf c}_\varepsilon$ admits an analytic continuation, we deduce $h_i(\varepsilon)$ is defined and is analytic on $(0-\updelta, 1+\updelta)$.
\par Suppose $\mathrm{val}(\varepsilon)$ has infinitely many pieces on $[0, 1]$ then, $\mathcal{S}$ has at least one of the points $0$ or $1$ as its accumulation point. This would again imply that the functions $h_i(\varepsilon)$ -- that are now analytic in $(0-\updelta, 1+\updelta)$ all must coincide which contradicts the non-triviality of the perturbation. 
\end{proof}
\begin{remark}\phantom{}\hspace{-1pt}\textit{Notice that in Theorem~\ref{thm:anal-ext}, we are not assuming the problem is even well-posed for $\varepsilon>1$ or for $\varepsilon<0$.} \phantom{}
\end{remark}
%-----------------------------------------------
%-----------------------------------------------
\subsubsection{\small \bf \textsf{ Upper bounds on the number of pieces}}\label{sec:ub}
\par In general, for an analytic perturbation, we can not bound the number of finite analytic pieces; however, if we assume the perturbation is given by  polynomials in $\varepsilon$ then, an upper bound can be estimated. 
\begin{theorem}\label{thm:poly-bound}
	Let ${\bf c}_\varepsilon$ satisfy the hypotheses of Theorem~\ref{thm:piecewise-anal} and furthermore,  
	\[
  {\bf c}_{\varepsilon} = \left(  p_1(\varepsilon), \cdots, p_{n}(\varepsilon)    \right), \quad 0\le \varepsilon\le 1,
	\]
	where $p_i(\varepsilon)$ is a polynomial of degree $d_i$.  Then, $\mathrm{val}(\varepsilon)$ has finitely many pieces and the number of pieces is bounded above by 
	\[
	d_{\max} \cdot \Uplambda(n,m) + 1.
	\]
	in which $m$ is the number of facets of ${\mathfrak Q}$. 
\end{theorem}
\begin{proof}[\footnotesize \textbf{Proof}] 
	Since the perturbation is given by polynomials, it -- in particular -- admits analytic continuation to a larger interval; hence, by Theorem~\ref{thm:anal-ext}, we know $\mathrm{val}(\varepsilon)$ has finitely many analytic pieces on $[0,1]$. These  pieces are indeed polynomials since $h_i = {\bf c}_{\varepsilon} \cdot \hat{v_i}$ is a polynomial for all $i$. 
		\par The number of pieces is equal to $|\mathcal{S} \cap (0,1)| + 1$.  On the other hand, from the proof of Lemma~\ref{lem:finite-switch}, one knows 
	$\mathcal{S} \subset \mathcal{A}$. Let $d_{\max} := \max d_i$; then, $h_i$ is a polynomial of degree  at most $d_{\max}$. The graphs of distinct two such polynomials can have no more than $d_{\max}$ points of intersection; hence, we deduce 
	\[
	\left| \mathcal{A} \right| \le d_{\max} \cdot \left| \mathcal{V}  \right|.
	\]
	\par By our Theorem~\ref{thm:ub}, $ \left| \mathcal{V}  \right| \le \Uplambda(n, m)$ holds where $m$ is the number of facets of the constant feasible set ${\mathfrak Q}$. Therefore, we get
\[
|\mathcal{S} \cap (0,1)| + 1 \le  d_{\max} \cdot \Uplambda(n,m) + 1.
\]
\end{proof}
\begin{corollary}
	If ${\bf c}_\varepsilon$ is affine in $\varepsilon$, then the number of affine pieces of the optimal value function is bounded above by $\Uplambda(n,m) + 1$.
\end{corollary}
%-----------------------------------------------
%-----------------------------------------------
% Section: Edge Ricci curvatures
%-----------------------------------------------
%-----------------------------------------------
\section{Some operator theoretic tools}\label{sec:oper-theo}
%-----------------------------------------------
%-----------------------------------------------
%-----------------------------------------------
\subsubsection*{\small \bf \textit{\textsf{Notation}}}
Throughout the rest of these notes, $\mathcal{C}_{\sf fs}(\G)$ denotes the space of compact support (continuous) functions with respect to the combinatorial distance; notice these are functions with finite support. $\LL^2(\G, \m)$ is the space of square $\m$-summable functions and $\ell^2(\G)$ is the space of square summable functions.  In this section, $\dist$ is an arbitrary distance function on $\G$. 
\begin{definition}[self-adjoint property]
	An operator $\mathcal{A}: \LL^2(\G,\m) \to \LL^2(\G,\m)$
	is self adjoint if
	\[
	\left< \mathcal{A}f, g  \right> = \left< f, \mathcal{A}g \right>, \quad \forall f,g \in \LL^2(\G, \m),
	\]
	where the inner product is
	\[
	<f,g> := \int_{\G} fg d\m = \sum_{z\in \G} f(z)g(z)\m(z).
	\]
	\par Operator  $\mathcal{A}:  \LL^2(\G)  \supset \mathsf{Dom}(\mathcal{A}) \to \LL^2(\G)$ is said to be essentially self-adjoint if it has unique self-adjoint extension to $\LL^2(\G,\m)$. 
	\par We say an operator $\mathcal{A}$ has the self-adjoint property if $\mathcal{A}$ is essentially self-adjoint and $\mathcal{C}_{\sf fs}(\G) \subset \mathsf{Dom}(\mathcal{A})$. 
\end{definition}
\begin{definition}[rough comparison principle]
 An operator $\mathcal{A}:  \LL^2(\G, \m) \supset \mathsf{Dom}(\mathcal{A})  \to \LL^2(\G, \m)$ is said to satisfy a rough comparison principle with range $R\ge0$ on a subspace $\mathbb{W} \subset  \LL^2(\G, \m)$, if there exists a constant $C_1$ such that for any $f$ in $\mathbb{W} \cap \Lip(1)$, it holds
	\[
	\text{$f$ has a global minimum at $x$}   \Longrightarrow \mathcal{A} f(z) \ge C_1,\quad  \forall z \in \B_{R}(x).
	\]
\end{definition}
\begin{definition}[two-sided rough comparison]\label{def:two-sided}
	The operator $\mathcal{A}:  \LL^2(\G, \m) \supset \mathsf{Dom}(\mathcal{A})  \to \LL^2(\G, \m)$ is said to satisfy a  two sided rough  rough comparison principle on a subspace $\mathbb{W} \subset  \LL^2(\G, \m)$, if there exists a function $C(x)$ such that  the following hold for $f \in \mathbb{W} \cap \Lip(1)$ 
	\[
	\text{$f$ has a global minimum at $x$}   \Longrightarrow \left| \mathcal{A} f(x)\right| \le C(x).
	\]
\end{definition}
\begin{definition}[rough differential operator]
		Suppose $\m \in \ell^1(\G)$. An operator $\mathcal{A}: \LL^2(\G, \m) \supset \mathsf{Dom}(\mathcal{A})  \to \LL^2(\G, \m)$ is said to be a rough differential operator whenever its kernel contains $1$ ($\iff$ all constants).  Notice  hypothesis  $\m \in \ell^1(\G)$ guarantees $1 \in  \LL^2(\G, \m)$.  So these operators are translation invariant. 
\end{definition}
\begin{definition}[weakly of divergence type]
	An operator $\mathcal{A}:  \LL^2(\G, \m) \supset \mathsf{Dom}(\mathcal{A})  \to \LL^2(\G, \m)$ is said to be weakly of divergence type whenever  $\mathcal{C}_{\sf fs}(\G) \subset \mathsf{Dom}(\mathcal{A})$ and
	\[
	\int_\G\mathcal{L} f d \m= 0, \quad \forall f \in \mathcal{C}_{\sf fs}(\G).
	\]
\end{definition}
\begin{definition}[semi-local operator]
	Let  $\G$ be a locally $\dist$-finite graph. An operator  $\mathcal{A}:  \LL^2(\G, \m) \supset \mathsf{Dom}(\mathcal{A})  \to \LL^2(\G, \m)$  is said to be semi-local with respect to $\dist$ with range $R$ in the subspace $\mathbb{W}$, if there exists $C_2>0$ such that 
	\[
	f\restr_{\B_{R}(x)} = 0 \Longrightarrow {\mathcal A}f(x) = 0, \quad \forall f \in \mathbb{W}\cap \Lip(1).
	\]
	and
	\[
	\left| \B_{2R}(x)  \right| \le C_2, \quad \forall x.
	\]
An operator is said to be semi-local whenever it is semi-local with some range $R>0$. 
\end{definition}
\begin{proposition}[two-sided comparison]\label{prop:st-comp}
Let $\G$ be a locally $\dist$-finite graph. Suppose $\mathcal{A}:  \LL^2(\G, \m) \supset \mathsf{Dom}(\mathcal{A})  \to \LL^2(\G, \m)$ is an operator with the following properties
\begin{enumerate}
	\item $\Lip(1) \subset \mathsf{Dom}(\mathcal{A}) $;
	\item $\mathcal{A}$ is semi-local with range $R>0$;
	\item  $\mathcal{A}$ is a  rough differential operator;
	\item $\mathcal{A}$  satisfies a rough comparison principle on $\Lip_1(\G)$ with range $2R$;
	\item there exists $C_3>0$ such that for all $z$, $\left| \mathcal{A} \dist(z,\cdot)(w) \right| \le C_3$ holds $\forall w \in \B_{2R}(z)$ .
\end{enumerate}
	Then, there exists $C>0$ depending on $C_1,C_2,C_3$, $\sup m\restr_{\B_{2R}\left( x_0\right)}$ such that
	\[
	\text{$f$ has a global minimum at $x$}   \Longrightarrow -\mathcal{A} f\left(x\right) \ge C\left( C_1,C_2,C_3,  \sup \m\restr_{\B_{2R}\left( x\right)} \right).
		\]
\end{proposition}
\begin{proof}[\footnotesize \textbf{Proof}]
	Suppose $f \in \Lip(1)$ attains a global minimum at $x$; since $\mathcal{A}$ is a rough differential operator, we can assume $f\ge 0$ and $f(x) = 0$. 
	Take the cut-off function 
	\[
	\upchi(z) := \left[ 2s - \dist(x,z)   \right]_+, \quad s := \sup\limits_{\B_{R}(x)} \dist(x,\cdot);
	\]
	set $\bar{f} := \upchi \wedge f$. It is easy to see that $\nicefrac{1}{4} (\upchi \wedge \bar{f})$ is a $1$-Lipschitz function.  We have 
	\[
	\upchi(z) \ge s \ge f(x,z), \quad  \forall z  \in \B_{R}(x),
	\]
	so 
	\[
	\nicefrac{1}{4}\bar{f} \restr_{ \B_{R}(x)} = \nicefrac{1}{4}f\restr_{ \B_{R}(x)},
	\]
	and notice both $\nicefrac{1}{4}\bar{f}$ and $\nicefrac{1}{4}f$ are $1$-Lipschitz functions. 
	\par By the semi-locality hypothesis, we get
	\[
	\mathcal{A} \bar{f}\left( x_0  \right) = \mathcal{A} f\left( x_0  \right).
	\]
\par For $z \not\in \B_{2R}\left( x_0  \right)$, the semi-locality gives $A\bar{f}(z) = 0$.  Now being weakly of divergence form implies
	\[
	\mathcal{A} \bar{f}\left( x_0  \right)\m\left( x_0  \right) + \sum_{z \in \B_{2R}\left( x_0  \right)} \mathcal{A} \bar{f}(z)\m(z) = 0,
	\]
	hence, we deduce
	\[
	-\mathcal{A} \bar{f}\left( x_0  \right)\m\left( x_0  \right)  =  \sum_{z \in \B_{2R}\left( x_0  \right)} \mathcal{A} \bar{f}(z) \m(z).
	\]
\par Now clearly $\nicefrac{1}{8}\left( \bar{f}(z) + 4\dist(x_0,z)\right) \ge 0$ is $1$-Lipschitz and  has a global minimum at $z=x_0$.  So by the rough comparison principle with range $2R$, one arrives at
		\begin{align*}
		\sum_{z \in \B_{2R}(x)} \mathcal{A} \bar{f}(z)\m(z) &\ge 8C_1-  4\sum_{z \in \B_{2R}(x)} \m(z)  \left| \mathcal{A} \left( \dist(x_0,\cdot) \right)(z)  \right| \\ &\ge  8C_1 - 4C_2C_3\left(\sup \m\restr_{\B_{2R}}\right) := C.
	\end{align*}
So clearly,
\[
- \mathcal{A} f\left( x_0  \right) = -\mathcal{A} \bar{f}\left( x_0  \right) \ge C\left( C_1,C_2,C_3,  \sup \m\restr_{\B_{2R}\left( x_0  \right)} \right),
\]
which is the desired conclusion. 
\end{proof}
\begin{definition}[good operators]\label{def:good-op}
The operator $\mathcal{A}:  \LL^2(\G, \m) \supset \mathsf{Dom}(\mathcal{A})  \to \LL^2(\G, \m)$ is said to be a  good operator if it satisfies the following properties
	\begin{enumerate}
	\item $\mathcal{A}$ satisfies self-adjoint property;
	\item $\mathcal{A}$ is a rough differential operator;
	\item $\mathcal{A}$ is weakly of divergence form.
\end{enumerate}
\end{definition}
\begin{remark}\phantom{}\hspace{-1pt}\textit{Except for the definition of the self-adjoint property and consequently for the definition of good operators, all other definitions also make sense for more general operators  $\mathcal{A}:  \R^\G \supset \mathsf{Dom}(\mathcal{A})  \to \R^\G$ that we will deal with in~\textsection\thinspace\ref{sec:op-theo-ric}.} \phantom{}
\end{remark}
%-----------------------------------------------
%-----------------------------------------------
% Section: Edge Ricci curvatures
%-----------------------------------------------
%-----------------------------------------------
\section{Discrete Ollivier-Ricci curvature}\label{sec:edge-ric}
%-----------------------------------------------
%-----------------------------------------------
%-----------------------------------------------
\subsection{Discrete-time Ollivier-Ricci curvature}\label{subsec:olric}
\par Discrete-time Ollivier coarse Ricci curvature, $\olric(x,y)$, of a random walk $\left\{ \upmu_z\right\}_{z\in X}$ (here, a random walk is a measurable point dependent measure $\upmu_z$ on a polish metric space $X$ with finite first moments) along $xy$ ($x,y$ are two points in the space), is defined as the ``\emph{deviance}'' of the ``\emph{ratio of $L^1$-Wasserstein distance of the corresponding probabilities and the distance of $x$ and $y$}'' from $1$~\cite{Ol}. 
\par In the setting of Riemannian manifolds with $\Ric \ge K$ -- and by using $\varepsilon$-step random walks characterized by uniform measures $\upmu^\varepsilon_z$ on $\B_\varepsilon(z)$ -- the Ollivier coarse curvature possesses nice precise asymptotics depending on distance and Ricci  curvature; indeed, when $d(x,y) \approx \varepsilon$, the Ollivier coarse Ricci curvature satisfies
\begin{align}\label{eq:olric-asymp}
	\olric_{\varepsilon}(x,y) = \nicefrac{\varepsilon^2}{2(N+2)} \, \Ric\left( \nicefrac{\vec{xy}}{\|\vec{xy}\|}\right)+ \mathcal{O}\left(\varepsilon^3\right);
\end{align}
here, $\vec{xy}$ denotes the vector in the tangent space solving $\exp_x{\vec{xy}} = y$~\cite{Ol}. 
\par In the setting of continuous-time Markov processes, one can assume the measures depend on time and represent how the delta measures diffuse in space (according to some differential equation $\nicefrac{\partial}{\partial t} u = \mathcal{L} u$). Then according to~\cite{Ol}, taking the derivative in the Ollivier coarse Ricci curvature, one computes the continuous-time Ollivier coarse Ricci curvature to be
\begin{align*}
{^{\mathcal{CTO}}\Ric}(x,y) = \text{-}\, \nicefrac{d}{d\varepsilon}\restr_{\varepsilon=0} \; \nicefrac{\Was_1\left( \upmu_x^\varepsilon ,  \upmu_y^\varepsilon \right)}{\dist_X(x,y)} =  \nicefrac{d}{d\varepsilon}\restr_{\varepsilon=0} \olric_\varepsilon;
\end{align*}
this indeed computes the first order (rescaled) limit of the discrete-time curvature. In the original paper~\cite{Ol}, this definition is presented contingent upon the the existence of such a limit. That is an important issue which we wish to address -- in the discrete setting -- in these notes. 
\begin{remark}\phantom{}\hspace{-1pt}\textit{If we take the derivative at $\varepsilon = 0$ in \eqref{eq:olric-asymp}, we get zero; meaning infinitesimally speaking, random walking on a manifold or the Euclidean space are the same i.e. Riemannian manifolds are infinitesimally Euclidean to the second order or the infinitesimal walks do not see the curvature. }
\par \textit{So in the continuous spaces, Ollivier coarse $\varepsilon$-step  curvature is more natural to use than the limit version. Or even better, one can use a suitable second derivative version instead e.g.
	\[
	\olric_{\mathsf{Riem}} (x,y) :=  \text{-}\, \nicefrac{d^2}{d\varepsilon^2}\restr_{\varepsilon=0} \; \nicefrac{\Was_1\left( \upmu_x^\varepsilon,  \upmu_y^\varepsilon \right)}{\dist(x,y)}.
	\]
which at least in the Riemannian setting -- and up to a dimensional multiple of -- is comparable to  the Ricci curvature lower bound; see the  asymptotics in~\cite{vRS}. We will shortly see that the second derivative version gives zero when applied in graphs and for  time-affine walks. So on  graphs the second order geometry is trivial for this particular type of random walks.}
\phantom{}
\end{remark}
\par On graphs, both the discrete-time and continuous-time Ollivier-Ricci curvatures can be made sense of, exactly as in the continuous setting. Of course these curvatures depend on the random walks that one chooses and for the continuous-time discrete Ollivier-Ricci curvature, existence of the defining derivative must be proven. 
%-----------------------------------------------
%-----------------------------------------------
%-----------------------------------------------
\subsection{Discrete-time Ollivier-Ricci curvature}\label{sec:dt-olric}
\par Given a random walk $\upmu_z^\varepsilon$ with finite first moment, and using a given distance $\dist$, the discrete-time Ollivier-Ricci curvatures are defined as
\[
\olric_\varepsilon (x,y) := 1 - \nicefrac{\Was_1\left( \upmu_x^\varepsilon ,  \upmu_y^\varepsilon  \right)}{\dist(x,y)},
\]
in which, the $L^1$-Wasserstein (or Kantorovich-Rubinstein) distance is
\begin{align*}
	\Was_1 \left(\upmu_1, \upmu_2 \right) := \inf_{q} \sum_{x,y\in \G} \dist(x,y) q(x,y),
\end{align*}
where the infimum is taken over all couplings $q \in {\Prob} (\G \cart \G)$ of $\upmu_1$ and $\upmu_2$; coupling means $\upmu_1$ and $\upmu_2$ are marginals of $q$.  The above is well-defined since $L^1$-Wasserstein distance between measures with finite first moment is well-defined. In the sequel, we always assume finite first moments. 
%-----------------------------------------------
%-----------------------------------------------
\subsubsection{\small \bf \textsf{Localization}}
\begin{lemma}\label{lem:loc-plan}
	For every $x$,  set $\Upomega^\varepsilon_{x}:= \mathrm{supp}\left( \upmu_x^\varepsilon\right)$. Then, there exists an optimal coupling $q$  of $\Upomega^\varepsilon_{x}$ and $\Upomega^\varepsilon_{y}$ supported in $\Upomega^\varepsilon_{xy} \times \Upomega^\varepsilon_{xy}$ where $\Upomega^\varepsilon_{xy} := \Upomega^\varepsilon_x \cup \Upomega^\varepsilon_y$. 
\end{lemma}
\begin{proof}[\footnotesize \textbf{Proof}]
This is straightforward from the fact that the first and second marginals of the couplings are supported in $\Upomega^\varepsilon_x$ and $\Upomega^\varepsilon_y$ resp. Hence, the couplings must be supported in $\Upomega^\varepsilon_{xy}$. The standard existence of optimal plans theorem \cite[Theorem 4.1]{Vil}. 
\end{proof}
\begin{theorem}
For finite step walks, $\olric_\varepsilon (x,y)$ coincides with the discrete-time Ollivier-Ricci curvature on the graph $\Upomega^\varepsilon_{xy}$ and is the optimal value in a finite LP problem.
\end{theorem}
\begin{proof}[\footnotesize \textbf{Proof}]
This is a direct consequence of Lemma~\ref{lem:loc-plan}. 
\end{proof}
%-----------------------------------------------
%-----------------------------------------------
\subsubsection{\small \bf \textsf{A particular form of random walks}}\label{sec:theta-walks}
\par A particular yet still rather general class of $1$-parameter random walks are
\[
\uptheta_z^\varepsilon(w) := \begin{cases} \nicelambda_1(\varepsilon) +  \nicelambda_2(\varepsilon)\varphi_1(z) & \text{if} \quad w=z \\ \nicelambda_3(\varepsilon) \varphi_2(z,w)& \text{if} \quad w \neq z \end{cases}, \quad \text{$\nicelambda_i$ is continuous for $i=1,2,3$.}
\]
Of course since these are probability measures, we need the normalization
\[
\nicelambda_1(\varepsilon) +  \nicelambda_2(\varepsilon) \varphi_1(z) + \sum_{w\neq z} \nicelambda_3(\varepsilon) \varphi_2(z,w) = 1,
\]
and the constraints
\[
\nicelambda_1(\varepsilon) +  \nicelambda_2(\varepsilon) \varphi_1(z), \nicelambda_3(\varepsilon) \varphi_2(z,w) \ge 0;
\]
and we need the convergence to $\updelta_x$ as $\varepsilon \downarrow 0$ so,
\[
\lim_{\varepsilon \downarrow 0} \nicelambda_3(\varepsilon) \varphi_2(z,w) = 0,
\]
is also needed. We call these, $\uptheta$-walks. Notice there is no assumption on symmetricity of $\varphi_2$. 
\begin{example}\phantom\hspace{-2pt}The $\uptheta$-walks provides an important class of random walks on graphs. Upon setting $\nicelambda_1(\varepsilon) \equiv 1$, $\nicelambda_2(\varepsilon) \equiv -1$, $\nicelambda_3 = \varepsilon$, $\varphi_1(z) = \Deg(z)$, and 
	\[
	\varphi_2(z,w) = \begin{cases} \nicefrac{\omega_{xy}}{\m(z)}  &   \text{if} \quad w \sim z \\ 0 &   \text{if} \quad  w \not\sim z \end{cases},
	\]
	one retrieves the random walks 
	\[
	\upbeta_z^\varepsilon(w) := \updelta_z(w) + \varepsilon \Delta \updelta_{z}(w) = \begin{cases} 1 - \varepsilon\Deg(z) & \text{if} \quad w=z \\  \nicefrac{\varepsilon \omega_{zw}}{\m(z)} & \text{if} \quad w\sim z \\ 0 &  \text{otherwise} \end{cases},
	\]
	which were considered in~\cite{JM}.  Recall $\Deg(z)$ is computed using the edge-weights $\omega$. Furthermore, the spacial case $\dist = \dist_\omega$ is considered in~\cite{MW}.  We call these $\upbeta$-walks. 
\par Upon setting $\nicelambda_1(\varepsilon) \equiv \varepsilon$, $\nicelambda_2(\varepsilon) \equiv 0$, $\nicelambda_3 = 1- \varepsilon$, and 
	\[
	\varphi_2(z,w) = \begin{cases} \deg(z)  &   \text{if} \quad w \sim z \\ 0 &   \text{if} \quad  w \not\sim z \end{cases},
	\]
	one retrieves the random walks 
	\[
	\upzeta_z^\varepsilon(w) := \updelta_z(w) + \varepsilon \Delta \updelta_{z}(w) = \begin{cases} \varepsilon & \text{if} \quad w=z \\  \nicefrac{(1 - \varepsilon)}{\deg(z)} & \text{if} \quad w\sim z \\ 0 &  \text{otherwise} \end{cases},
	\]
	which were considered in~\cite{LLY}.  We call these $\upzeta$-walks. 
	\par Upon setting $\nicelambda_1(\varepsilon) \equiv \varepsilon$, $\nicelambda_2(\varepsilon) \equiv 0$, $\nicelambda_3 \equiv \nicefrac{(1- \varepsilon)}{C}$, and 
	\[
	\varphi_2(z,w) = \begin{cases} e^{-\dist(z,w)^p} &   \text{if} \quad w \sim z \\ 0 &   \text{if} \quad  w \not\sim z \end{cases},
	\]
	one retrieves the random walks 
	\[
	\upxi_z^\varepsilon(w) := \begin{cases} \varepsilon & \text{if} \quad w=z  \\  \nicefrac{(1-\varepsilon)e^{-\dist(z,w)^p}}{C} & \text{if} \quad w\sim z \\ 0 &  \text{otherwise} \end{cases},
	\]
	which were considered in~\cite{NLLG}. We call these $\upxi$-walks. These $1$-combinatorial step walks ($\upbeta, \upzeta, \upxi$)  are also known as lazy walks.
\end{example}
\begin{remark}[properties of $\uptheta$-walks]\phantom{}\hspace{-1pt}\textit{For $\uptheta$-walks, being time-affine, time-polynomial or time-analytic amounts to $\nicelambda_i$, $i=1,2,3$  being affine, polynomial or analytic (resp.) and admitting (affine, polynomial, analytic) continuations on $(-\updelta, 1+\updelta)$ for some small $\updelta>0$.  Of course the first two are special cases of the third one and they always admit affine or polynomial continuations. }
\par  \textit{Also for $\uptheta$-walks, being a local walk is equivalent to $\varphi_2$ being of bounded support.} \phantom{}
\end{remark}
%-----------------------------------------------
%-----------------------------------------------
\subsubsection{\small \bf  \textsf{Curvature concavity for time-affine local walks}}
\begin{theorem}[concavity]\label{thm:conc}
Let $\upmu_z^\varepsilon$ be a time-affine walk, then $\olric_\varepsilon (x,y)$ is a concave function of $\varepsilon$.  
\end{theorem}
\begin{proof}[\footnotesize \textbf{Proof}]
The Kantorovich dual formulation of Wasserstein distance \eqref{eq:kant} applied to linear walks, gives the Wasserstein distance as the supremum of a family of lines, so it must be convex since since  supremum of a suitable family of convex functions is convex. Now let us provide a rigorous proof. 
\par It is straightforward to see that for all $\varepsilon$, all couplings form a convex polytope $\mathfrak{P}_\varepsilon$ (not necessarily bounded) in  $\R^{\aleph_0}$ (sometimes called the transportation polytope).  Due to the affinity of $\upmu^\varepsilon_z(w)$ in $\varepsilon$, one immediately deduces the geometric concavity property
	\[
	(1-t)\mathfrak{P}_{\varepsilon_1} + t \mathfrak{P}_{\varepsilon_1} \subset  \mathfrak{P}_{(1-t)\varepsilon_1 + t\varepsilon_2}.
	\]
	
In particular, letting $q^{\varepsilon_1}$, $q^{\varepsilon_2}$ be optimal couplings (for existence of optimal couplings, e.g. see~\cite[Theorem 4.1]{Vil}) for $\upmu_x^{\varepsilon_1}, \upmu_y^{\varepsilon_1}$  and  $\upmu_x^{\varepsilon_2}, \upmu_y^{\varepsilon_2}$ resp.; then
\[
\bar{q} := (1-t) q^{\varepsilon_1} + tq^{\varepsilon_2},
\]	
is a coupling for two measures $\upmu_{x}^{(1-t)\varepsilon_1 + t\varepsilon_2}$ and $\upmu_{y}^{(1-t)\varepsilon_1 + t\varepsilon_2}$. Using ${\bar q}$ as a test coupling, we deduce
\[
\Was_1\left( \upmu_{x}^{(1-t)\varepsilon_1 + t\varepsilon_2}, \upmu_{y}^{(1-t)\varepsilon_1 + t\varepsilon_2} \right)\le (1-t) \Was_1\left( \upmu_x^{\varepsilon_1}, \upmu_y^{\varepsilon_1} \right) + t\Was_1\left( \upmu_x^{\varepsilon_2}, \upmu_y^{\varepsilon_2} \right),
\]
which is the convexity of $\Was_1$ in $\varepsilon$. This -- directly from the definition -- implies concavity of $\olric_\varepsilon (x,y)$ in $\varepsilon$  ; namely, 
\begin{align}\label{eq:ric-conc}
\olric_{(1-t)\varepsilon_1 + t\varepsilon_2} (x,y) \ge (1-t) \olric_{\varepsilon_1} (x,y) + t \olric_{\varepsilon_2} (x,y).
\end{align}
\end{proof}
\begin{remark}\phantom{}\hspace{-1pt}\textit{This concavity does not necessarily hold for non time-affine random walks. However one expects an infinitesimal version of this concavity (as well as affinity) to hold for all walks (think of a concavity for the linear approximation of random walks); indeed -- thanks to a general limit-free formulation -- such infinitesimal version of concavity holds for the Ollivier-Ricci curvature in the form of a concavity relation in terms of the generating operators  $\mathcal{L}$ in the  operator theoretic Ollivier-Ricci curvature $\olric_{\mathcal{L}}$; see~\textsection\thinspace\ref{prop:concav-op}. }\phantom{}
\end{remark}
\begin{corollary}
The concavity in $\varepsilon$ holds for $\upbeta$, $\upzeta$ and $\upxi$-walks as well as for all other time-affine $\uptheta$-walks .
\end{corollary}
%-----------------------------------------------
%-----------------------------------------------
\subsubsection{\small \bf \textsf{Localization in the dual problem}}
\begin{theorem}\label{thm:loc-dual}
	Let $\upmu_z^\varepsilon$ be a finite-step random walk. then
	\begin{align*}
		\Was_1\left( \upmu_x^\varepsilon, \upmu_y^\varepsilon \right) = \sup_{\substack{f: \Upomega^\varepsilon_{xy} \to \R \\ f \in \Lip(1)}}  \sum f(z) \left( \upmu_x^\varepsilon(z) - \upmu_y^\varepsilon(z)  \right);
	\end{align*}
recall $\Upomega^\varepsilon_{xy} := \Upomega^\varepsilon_x  \cup \Upomega^\varepsilon_y $. 
\end{theorem}
\begin{proof}[\footnotesize \textbf{Proof}]
	Consider the numbers
	\[
	s_x := \max_{z \in \Upomega_x} \dist(x,z), \quad 	s_y := \max_{w \in \Upomega_y} \dist(y,z),
	\]
	and set
	\[
	s:= s_x \vee s_y + \dist(x,y).
	\]
	\par By Kantorovich's duality~\cite[Chapter 5]{Vil}, it follows
	\begin{align}\label{eq:kant}
		\Was_1\left( \upmu_x^\varepsilon, \upmu_y^\varepsilon \right) = \sup_{f \in \Lip(1)}  \sum f(z) \left( \upmu_x^\varepsilon(z) - \upmu_y^\varepsilon(z)  \right);
	\end{align}
	so, we will instead look for a localized minimizer $f$ for the dual problem.  
	\par Translating $f$ by a constant would not affect the optimal value in \eqref{eq:kant}, hence, we can assume $f(x) = 0$. Take the cut-off function
	\[
	\upchi(z) := \left[ s \wedge \left( 2s - \dist(x,z)  \right)   \right]_+, 
	\]
\par 	It is easy to check that $\upchi.$, is $1$-Lipschitz w.r.t. $\dist_\eta$ and  equals $s$ on $\Upomega^{\varepsilon}_{xy}$.  Hence, we deduce $\upchi \wedge f$ coincides with $f$ on $\Upomega^{\varepsilon}_{xy}$; indeed, for $z \in \Upomega^{\varepsilon}_y$, one easily gets
	\[
	f(z) \le f(y) +  s_y \le \dist(x,y) + s_y \le s.
	\] 
	Now set
	\[
	\tilde{f} := \left( -\upchi \vee f  \right) \wedge \upchi.
	\]
 Clearly, $\tilde{f}$ coincides with $f$ on $\Upomega^{\varepsilon}_{xy}$ so in particular,  $\tilde{f}$ is $1$-Lipschitz within $\Upomega_{xy}$ (however it is $4$ Lipschitz at best over $\G$); furthermore,  $\tilde{f}$ is supported in $\B^\eta_{2s}(x)$. Since $\upmu_x^{\varepsilon}$ and $\upmu_y^{\varepsilon}$ are  supported in $\Upomega^{\varepsilon}_{xy}$, we deduce
	\[
	\sum f(z) \left( \upmu_x^\varepsilon(z) - \upmu_y^\varepsilon(z)  \right) = \sum \tilde{f}\restr_{\Upomega^{\varepsilon}_{xy}}(z) \left( \upmu_x^\varepsilon(z) - \upmu_y^\varepsilon(z)  \right).
	\]
	Notice $\tilde{f}\restr_{\Upomega^{\varepsilon}_{xy}}$	is a $1$-Lip function on $\Upomega^{\varepsilon}_{xy}$. 
	
\par So the problem becomes a finite linear programming problem hence, admits a minimizer $\tilde{f}$ and by construction this minimizer is supported in   $\B_{2s}(x)$.  This means 
	\begin{align*}
		\Was_1\left( \upmu_x^\varepsilon, \upmu_y^\varepsilon\right) = \sup_{ \substack{f\restr_{\Upomega_{xy}} \in \Lip(1)  \\  \mathrm{supp}\left( f \right) \subset  \B^\eta_{2s}(x) }}  \sum f(z) \left( \upmu_x^\varepsilon(z) - \upmu_y^\varepsilon(z)  \right);
	\end{align*}
To furthermore localize to $\Upomega^{\varepsilon}_{xy}$, suppose $f: \Upomega^{\varepsilon}_{xy} \to \R$ is $1$-Lipschitz, then the standard construction 
	\[
	\hat{f}(z) := \sup_{w \in \B^\eta_{2s}(x)} \left( f(w) - \dist(z, w)  \right),
	\]
	produces a $1$-Lipschitz extension of $f$ over $\B_{2s}(x)$. This implies we can further localize the minimizer to be a $1$-Lipschitz function with domain $\Upomega^{\varepsilon}_{xy}$. 
\end{proof}
%-----------------------------------------------
%-----------------------------------------------
\subsubsection{\small \bf \textsf{Curvature regularity for local walks}}
\begin{theorem}[regularity]\label{thm:PL}
	Suppose $\upmu_z^\varepsilon$ is a time-analytic local walk. For any fixed pair $x,y$, $\olric_\varepsilon (x,y)$ is a locally Lipschitz function in $(0,1)$ (in terms of $\varepsilon$) and is piece-wise analytic function of $\varepsilon$ with finitely many distinct analytic pieces in $[0,1]$. Similar statements also hold if we replace time-analytic with  time-affine or time-polynomial random walks and replace piece-wise analytic by piece-wise affine or piece-wise polynomial. 
\end{theorem}
\begin{proof}[\footnotesize \textbf{Proof}]
\par	Let us consider the Kantorovich's dual formulation 
\begin{align*}
	\Was_1\left( \upmu_x^\varepsilon, \upmu_y^\varepsilon \right) = \sup_{f \in \Lip(1)}  \sum f(z) \left( \upmu_x^\varepsilon(z) - \upmu_y^\varepsilon(z)  \right);
\end{align*} 
	since $\upmu_x^\varepsilon$ is a local walk, and by using localization (Theorem~\ref{thm:loc-dual}), we can assume $f$ is supported in $\mathcal{K}_{xy}$. So computing the $\LL^1$-Waserstein distance reduces to computing the optimal value ia linear programming problem on $\R^N$ with inequality constraints where $N = \left| \mathcal{K}_{xy} \right|$; indeed, the $f \in \Lip(1)$ constraint can be written as a collection of $2N(N-1)$ number of one dimensional linear inequalities
	\[
	f(y) - f(x) \le \dist(x,y) \quad \text{and} \quad  f(x) - f(y) \le \dist(x,y), \quad \forall x,y \in \mathcal{K}_{xy}.
	\]
Denote an optimal $f$ by $f_\varepsilon$.  Notice as $\varepsilon$ varies, the set of feasible $f$ does not change. 	
\par If we furthermore, add the constraint $f_\varepsilon(x) = \diam\left(\mathcal{K}_{xy} \right)$, the optimal value would not change since the  translations of optimal functions are also optimal (in the Kantorovitch's dual problem). This means the new set of feasible $f$ forms a bounded convex polytope $\mathfrak{Q}$ in the non-negative cone in $\R^N$ (meaning all components are non-negative). 
\par Also notice since the problem is finite, for every $\varepsilon$, the problem is well-posed and stays well-posed after perturbations of $\upmu_z^\varepsilon$, meaning the corresponding vectors ${\bf d}_\varepsilon$ are in the interior of the well-posed region. This means we can invoke the optimal value sensitivity results of~\textsection\thinspace\ref{sec:sens-anal}. 
\par Therefore, by Theorem~\ref{thm:piecewise-anal}, we deduce 	$\Was_1\left( \upmu_x^\varepsilon, \upmu_y^\varepsilon\right)$ is locally Lipschitz in $(0,1)$ and is piece-wise analytic in $(0,1)$ with finitely many distinct pieces in $[0,1]$. 
	\end{proof}
\begin{example}\phantom{}\hspace{-2pt}$\upbeta$, $\upzeta$ and $\upxi$-walks are time-affine local walks hence, by Theorem~\ref{thm:PL}, the corresponding discrete-time generalized Ollivier-Ricci curvature is a locally Lipschitz and piece-wise analytic function of $\varepsilon$ with finitely many pieces. 
\end{example}	
\begin{remark}\phantom{}\hspace{-1pt}\textit{Indeed, one shows the stronger result that the optimal value is locally Lipschitz in a slightly larger open interval so it is Lipschitz in $[0,1]$ by standard compactness arguments. }
\end{remark}
%-----------------------------------------------
%-----------------------------------------------
\subsubsection{\small \bf \textsf{Upper bound on the number of analytic (affine or polynomial) pieces}}
As we alluded to in~\textsection\thinspace\ref{sec:ub}, only knowing time-analyticity, we can not uniformly bound the number of distinct analytic pieces, however for time-polynomials walks, the algebraic and convex geometric nature of the problem, allows us to do so. 
\begin{theorem}[bound on the number of pieces]\label{thm:piece-bound}
	\par Let $\upmu_z^\varepsilon$ be a polynomial-time local walk. Then for each $x,y$, the number of distinct polynomial pieces of $\olric_\varepsilon$ is bounded above by
	\[
	\left( \max\limits_{z,w\in \mathcal{K}_{xy}} \deg \p_{zw}(\varepsilon)\right)  \cdot  \Uplambda(\mathcal{N}_{xy}, 2\, \mathcal{N}_{xy}^{\,^2}-2\, \mathcal{N}_{xy}+1).
	\]
	where $\Uplambda$ is the upper bound function obtained in Theorem~\ref{thm:ub}.  Recall $\mathcal{N}_{xy}= \left| \mathcal{K}_{xy}\right|$ .
\end{theorem}
\begin{proof}[\footnotesize \textbf{Proof}]
	By Theorem~\ref{thm:loc-dual} and by the extension of $1$-Lipschitz functions, for all $\varepsilon$, we have
		\begin{align*}
		\Was_1\left( \upmu_x^\varepsilon, \upmu_y^\varepsilon \right) = \sup_{\substack{f: \mathcal{K} \to \R \\ f \in \Lip(1)}}  \sum f(z) \left( \upmu_x^\varepsilon(z) - \upmu_y^\varepsilon(z)  \right),
	\end{align*}
which is an LP problem in $\R^{\mathcal{N}_{xy}}$ with $2\,\mathcal{N}_{xy}(\mathcal{N}_{xy}-1)$ number of constraints.  To get non-negativity, we add another constraint $f(x) = \diam\left( \mathcal{K}_{xy} \right)$ (which does not affect the optimal value function). So in total, there are $2\, \mathcal{N}_{xy}(\mathcal{N}_{xy}-1) + 1 = 2\, \mathcal{N}_{xy}^{\,^2}-2\,\mathcal{N}_{xy}+1$ constraints. Notice boundedness of $\mathcal{K}_{xy}$ and $f$ being $1$-Lipschitz ensures the boundedness of the feasible set.  
\par The objective functions are polynomials with degree at most equal to $d_{\max} = \max\limits_{z,w\in \mathcal{K}_{xy}} \deg \p_{zw}(\varepsilon)$. So we can use Theorem~\ref{thm:poly-bound} to ge the conclusion. 
\end{proof}
\begin{example}\phantom\hspace{-2pt}For $\upbeta$, $\upzeta$ and $\upxi$-walks, the corresponding $\olric_\varepsilon$ has at most
	\[
	 \Uplambda(\mathcal{M}, 2\mathcal{M}^2-2\mathcal{M}+1),
	\]
	affine pieces where $\mathcal{M} := \left|  \B_{1}(x) \cup \B_1(y) \right|$. 
\end{example}
%-----------------------------------------------
%-----------------------------------------------
\subsubsection*{\small \bf \textit{\textsf{Open problem}}} Disprove the sharpness of the bound obtained in~Theorem~\ref{thm:piece-bound} or provide an example which shows sharpness.
\begin{remark}[the uni-modular case]\phantom{}\hspace{-1pt}\textit{We do not know whether the number of regular pieces obtained in Theorem~\ref{thm:piece-bound} is sharp or not; However we know for combinatorial graphs  ($m\equiv \omega \equiv 1$ and $\dist = \dist_\omega$), the corresponding LP problem is a uni-modular problem with integer parameters and integer optimal values; even the vertices of $\mathfrak{Q}$ are lattice points. This causes a lot of rigidity and as a result, the upper bound on the number of pieces for $\zeta$-walks is equal to $3$~\cite{BCLMP}. This means our upper bound is \emph{not sharp} in the uni-modular case. }\phantom{}
\end{remark}
%-----------------------------------------------
%-----------------------------------------------
\subsection{Continuous-time Ollivier-Ricci curvature}\label{sec:ct-olric}
By Ollivier's definition, the continuous-time generalized Ollivier-Ricci curvature for a random walk $\upmu_z^\varepsilon$ is given by
\begin{align*}
	{^{\mathcal{CTO}}\Ric}(x,y) = \text{-}\, \nicefrac{d}{d\varepsilon}\restr_{\varepsilon=0} \; \nicefrac{\Was_1\left( \upmu_x^\varepsilon,  \upmu_y^\varepsilon \right)}{\dist(x,y)} = \lim_{\varepsilon \downarrow 0} \nicefrac{1}{\varepsilon} \olric_\varepsilon(x,y) = \text{-}\, \nicefrac{d}{d\varepsilon}\restr_{\varepsilon=0} \;  \olric_\varepsilon(x,y),
\end{align*}
\emph{provided the limit exists}. In this section, we will consider this definition and discuss its well-definition and other properties. For simplicity, we will just use the notation $\olric$ for the continuous-time curvature.
\par Also what really plays a role in $\olric$ is the germ of the random walk at $\varepsilon=0$, so the locality assumption that appear in this section only needs to be satisfied on some non-trivial interval containing  $\varepsilon = 0$; a point that we will not -- for the rest of this section -- further emphasize on.
%-----------------------------------------------
%-----------------------------------------------
\subsubsection{\small \bf \textsf{Restriction of Kantorovich duality to finite support functions }}
At some places, we will need to be able to restrict the maximization in Kantorovich duality to finite support test functions. The following theorem affords us the said restriction. 
\begin{theorem}
	For any two measures $\upmu$ and $\upnu$ with finite first moments, 
	\begin{align}\label{eq:kant-revisited}
	\Was\left( \upmu, \upnu \right) = \sup_{f\in \Lip(1) \cap  \mathcal{C}_{\sf fs}(\G)}  \sum f(z) \left(  \upmu(z)  -  \upnu(z) \right),
	\end{align}
holds true.
\end{theorem}
\begin{proof}[\footnotesize \textbf{Proof}]
	By Kanotorvitch duality, we know 
	\[
	\Was\left( \upmu, \upnu \right) = \sup_{f\in \Lip(1)}  \sum f(z) \left(  \upmu(z)  -  \upnu(z) \right);
	\]
	hence,
	\begin{align}\label{eq:1}
		\Was\left( \upmu, \upnu \right)  \ge \sup_{f\in \Lip(1) \cap  \mathcal{C}_{\sf fs}(\G)}  \sum f(z) \left(  \upmu(z)  -  \upnu(z) \right).
	\end{align}
\par Take an exhaustion $G_1 \subset G_2 \subset \cdots$  of $\G$ by finite subgraphs (for example by combinatorial balls) with $\upmu(G_1) , \nu(G_1) \neq 0$.  Let
	\[
	 \upmu_i := \nicefrac{1}{\upmu(G_i)} \upmu\restr_{G_i},\quad  \upnu_i := \nicefrac{1}{\upnu(G_i)} \upnu\restr_{G_i}.
	\]
Then clearly $\upmu_i \rightharpoonup \upmu$ and $\upnu_i \rightharpoonup \upnu$ (weakly) in the space of probability measures with first finite moments.  
As a result, $\upmu_i - \upnu_i  \rightharpoonup \upmu - \upnu$. By the continuity of Wasserstein distance w.r.t. weak convergence \cite[Corollary 6.11]{Vil}, we deuce
\[
\Was\left( \upmu_i, \upnu_i \right) \to \Was\left( \upmu, \upnu \right).
\]
Therefore 
\begin{align*}
	\max\left\{ \upmu(G_i), \upnu(G_i)  \right\}  \Was_1\left(\upmu, \upnu\right) &\ge \max\left\{ \upmu(G_i), \upnu(G_i)  \right\}  \sup_{\substack{f\in \Lip(1) \\ \mathrm{supp}(f) \subset G_i}}  \sum f(z) \left(  \upmu(z)  -  \upnu(z) \right)  \\ &\ge   \sup_{\substack{f\in \Lip(1) \\ \mathrm{supp}(f) \subset G_i}}  \sum f(z) \left(  \upmu_i(z)  -  \upnu_i(z) \right)  \\&= \Was_1\left( \upmu_i, \upnu_i\right).
\end{align*}

\par Upon taking the limit, it follows
\begin{align}\label{eq:2}
 \Was_1(\upmu, \upnu) &= \lim_{i\to \infty}  \sup_{\substack{f\in \Lip(1) \\ \mathrm{supp}(f) \subset G_i}}  \sum f(z) \left(  \upmu_i(z)  -  \upnu_i(z) \right)  \\ &\le  \sup_{f\in \Lip(1) \cap  \mathcal{C}_{\sf fs}(\G)}  \sum f(z) \left(  \upmu(z)  -  \upnu(z) \right).\notag
\end{align}
Combining \eqref{eq:1} and \eqref{eq:2} gives the conclusion. 
\end{proof}
%-----------------------------------------------
%-----------------------------------------------
\subsubsection{\small \bf \textsf{Existence of $\olric$ for pleasant walks}}
 For more insight, we prove the case of local-walks separately.
\begin{theorem}
	$\olric(x,y)$ is well-defined for all time-analytic local-walks.
\end{theorem}
\begin{proof}[\footnotesize \textbf{Proof}]
	\par Let $\upmu_z^\varepsilon$ be a time-analytic local walk; so, by definition $\upmu_z^\varepsilon(y)$ admits continuation over $(-\updelta, 1+\updelta)$.  By Theorems~\ref{thm:PL}, we know the graph of $\olric_\varepsilon$ is comprised of finitely many analytic pieces on $[0,1]$.  Let $h_1$ be the first analytic piece defined on $(0,\updelta')$, then $h_1'(t)$ is also analytic hence, either $\lim_{t\downarrow 0} h_1'(t)$ is finite or otherwise  $\olric_\varepsilon$ must be unbounded as $\varepsilon \to 0$ which is a contradiction. So $\lim_{t\downarrow 0} h_1'(t)$ exists and due to analyticity and $h_1(0) = 0$, we have 
	\[
	\lim_{\varepsilon \downarrow  0}  h_1'(\varepsilon) =  \lim_{\varepsilon \downarrow  0} \nicefrac{1}{\varepsilon} \; h_1(\varepsilon) = \olric(x,y).
	\]
is finite.
\end{proof}
	\begin{corollary}
	$\olric(x,y)$ is well-defined for all time-analytic local $\uptheta$-walks.
\end{corollary}
\begin{example}\phantom\hspace{-2pt}$\upbeta$, $\upzeta$ and $\upxi$-walks give rise to well-defined continuous-time generalized Ollivier-Ricci curvatures.  So, in particular, we recover the well-definition results in~\cite{LLY, MW, JM}. 
\end{example}
\begin{theorem}\label{eq:wd-affine}
	$\olric(x,y)$ is well-defined for all time-affine walks of the form $\upmu_z^\varepsilon(y) = \updelta_z(y) + \varepsilon \upmu_z(y)$ where $\upmu_z$ is a zero-mass signed measure with finite first moment.
\end{theorem}
\begin{proof}[\footnotesize \textbf{Proof}]
By Theorem~\ref{thm:conc}, we know $\olric_\varepsilon$ is concave in $\varepsilon$ hence, it admits non-increasing (left and right) Dinni derivatives at every point. So 
\[
\lim_{\varepsilon \downarrow  0} \varepsilon^{-1} \olric_\varepsilon,
\] 
exists provided $ \varepsilon^{-1} \olric_\varepsilon$ is bounded from below as $\varepsilon \downarrow 0$. 
\par From the triangle inequality
\[
\Was_1\left( \upmu_x^\varepsilon, \upmu_y^\varepsilon\right) \le \Was_1\left( \updelta_x, \upmu_x^\varepsilon\right) + \Was_1\left( \updelta_x, \updelta_y\right) + \Was_1\left( \updelta_y, \upmu_y^\varepsilon\right),
\]
we deduce
\[
\nicefrac{1}{\varepsilon}\Big( 1 - \nicefrac{\Was_1\left( \upmu_x^\varepsilon, \upmu_y^\varepsilon\right)}{\dist(x,y)}  \Big) \ge   \nicefrac{1}{\varepsilon}\Big(  -  \nicefrac{\Was_1\left(\updelta_x, \upmu_x^\varepsilon\right)}{\dist(x,y)}   - \nicefrac{ \Was_1\left( \updelta_y, \upmu_y^\varepsilon\right)}{\dist(x,y)} \Big).  
\]
Therefore,  boundedness from below is ensured once 
\begin{align}\label{eq:bdd}
\nicefrac{d^+}{d\varepsilon} \Was_1\left( \updelta_z, \upmu_z^\varepsilon\right) = \mathcal{O}(\varepsilon).
\end{align}
\par From the hypotheses, one obtains
\[
 \Was_1\left( \updelta_z, \upmu_z^\varepsilon\right) \le \int \dist(z, w) d\upmu_z^\varepsilon(w)  = \int_{\G\smallsetminus \{z\}} \dist(z, y) d(\updelta_z + \varepsilon \upmu_z) = \varepsilon  \int_{\G} \dist(z, y) d\upmu_z = O(\varepsilon),
\]
which gives \eqref{eq:bdd} and the conclusion follows. 
\end{proof}
\begin{definition}[pleasant walks]\label{defn:pls-walk}
	A pleasant walk is a walk $\upmu_z^\varepsilon$ of the form
	\[
	\upmu_z^\varepsilon = \updelta_z + \varepsilon \upmu_z + \mathcal{R}^\varepsilon_z,
		\]
where 
$
 \mathcal{R}^\varepsilon_z = \varepsilon^2\mathcal{J}^\varepsilon \updelta_x(z),
$
that satisfies the following properties
\begin{enumerate}	
	\item  $\mathcal{J}^\varepsilon$ is a good operator;
	\item $\mathcal{J}^\varepsilon$ satisfies a rough comparison principle on the space $1$-Lipschitz functions with range $0$ and with a constant $C_\varepsilon \le O(\varepsilon^{-1 + \alpha_1} )$ as $\varepsilon \downarrow 0$ and for some $\alpha_1>0$;
	\item $\mathcal{J}^\varepsilon d(z,\cdot)(z) \le O(\varepsilon^{-1 + \alpha_2})$, $\forall z$ and $\forall \varepsilon$ and some $\alpha_2>0$.
	\item $\upmu_z$ is zero mass signed measures with finite first moments. 
\end{enumerate}
\end{definition}
\begin{theorem}\label{thm:pls-walk-0}
	$\olric$ is well-defined for pleasant walks. 
\end{theorem}
\begin{proof}[\footnotesize \textbf{Proof}]
	By (1) and (4), we know the $1$-jet of $\upmu_z^\varepsilon$
\[
{^1}\upmu_z^\varepsilon := \updelta_z + \varepsilon \upmu_z,
\]	
is also a probability measure.  
\par By Kantorovich's dual formulation~\eqref{eq:kant-revisited}, we get 
\begin{align*}
\Was\left( \upmu_z^\varepsilon, {^1}\upmu_z^\varepsilon \right) &= \sup_{f\in \Lip(1) \cap  \mathcal{C}_{\sf fs}(\G)}  \sum f(w) \left( \upmu_z^\varepsilon-  {^1}\upmu_z^\varepsilon  \right)(w)
\\&= \sup_{f\in \Lip(1) \cap  \mathcal{C}_{\sf fs}(\G)}   f(w) \left(  \varepsilon^2 \mathcal{J}^\varepsilon\updelta_z \right)(w)
\\&= \varepsilon^2 \sup_{f\in \Lip(1) \cap  \mathcal{C}_{\sf fs}(\G)}    \sum \updelta_z(w)\left(  \mathcal{J}^\varepsilon f  \right)(w)
\\&= \varepsilon^2  \sup_{f\in \Lip(1) \cap  \mathcal{C}_{\sf fs}(\G)}   \mathcal{J}^\varepsilon f(z).
\end{align*}
where the equality before last, holds by the self-adjoint property of $\mathcal{J}^\varepsilon$. 
\par Since $\mathcal{J}^\varepsilon$ is a rough differential operator, we can assume the normalization $f(z) = 0$. Therefore, we get $f(w) \le \dist(z,w)$. The function
\[
\nicefrac{1}{2} \left( \dist(z,\cdot) - f(\cdot)  \right),
\]
is a $1$-Lipschitz nonnegative function attaining a global minimim at $z$; hence, by rough the comparison principle with range $0$, one gets
\[
\nicefrac{1}{2} \mathcal{J}^\varepsilon \left(\dist(z,\cdot) - f(\cdot) \right) (z) \ge -C_\varepsilon,
\]
which means 
\[
\mathcal{J}^\varepsilon f(z) \le \mathcal{J}^\varepsilon \dist(z, \cdot)(z)  + 2C_\varepsilon \le \mathcal{O}\left( \varepsilon^{-1 + \min \alpha_i} \right);
\]
this in turn implies
\[
\Was\left( \upmu_z^\varepsilon, {^1}\upmu_z^\varepsilon\right) \le \mathcal{O}\left( \varepsilon^{1 + \min \alpha_i} \right),\quad \forall z.
\]
\par As a result,
\begin{align*}
&\left| \Was\left( \upmu_x^\varepsilon, \upmu_y^\varepsilon\right) - \Was\left( {^1}\upmu_x^\varepsilon, {^1}\upmu_y^\varepsilon \right) \right| = \left| \Was\left( \upmu_x^\varepsilon, \upmu_y^\varepsilon\right) \pm \Was\left( \upmu_x^\varepsilon, {^1}\upmu_x^\varepsilon \right)  - \Was\left( {^1}\upmu_x^\varepsilon ,{^1}\upmu_y^\varepsilon \right) \right| 
\\&\le \left|  \Was\left( \upmu_x^\varepsilon, \upmu_y^\varepsilon\right) -   \Was\left( \upmu_x^\varepsilon, {^1}\upmu_y^\varepsilon\right) \right|  + \left|  \Was\left( \upmu_x^\varepsilon, {^1}\upmu_y^\varepsilon \right)  - \Was\left( {^1}\upmu_x^\varepsilon ,{^1}\upmu_y^\varepsilon \right)   \right| \\
&\le \Was_1\left( \upmu_y^\varepsilon, {^1}\upmu_y^\varepsilon\right) + \Was_1\left( \upmu_x^\varepsilon, {^1}\upmu_x^\varepsilon \right)
\\&\le \mathcal{O} \left( \varepsilon^{1 + \min \alpha_i} \right).
\end{align*}
Thus,
\[
\left|  \nicefrac{1}{\varepsilon}\Big( 1 - \Was\left( \upmu_x^\varepsilon, \upmu_y^\varepsilon\right) \Big) - \nicefrac{1}{\varepsilon}\Big(  1 - \Was\left( {^1}\upmu_x^\varepsilon ,{^1}\upmu_y^\varepsilon \right) \Big) \right| =  \mathcal{O}\left( \varepsilon^{\min \alpha_i}\right);
\]
upon taking limit as $\varepsilon\downarrow 0$, we deuce
\[
\Ric(x,y) = {^1}\Ric(x,y);
\]
notice the existence of $ {^1}\Ric(x,y)$ (curvature w.r.t. the 1-jet random walk which is time-affine) was previously established in Theorem~\ref{eq:wd-affine}.
\end{proof}
\subsubsection{\small \bf \textsf{A criterion for existence of $\olric$ for general  Markovian walks}}
Suppose the continuous-time random walk is given by a time-homogeneous Markov kernel $p$ (the density of the random walk measure),
\[
\upmu_z^\varepsilon(w) := p(\varepsilon,z,dw);
\]
so the random walk (as a measure) is given by
\[
\upmu_z^\varepsilon (A) := \int_A p(\varepsilon,z,dw).
\]
\par The Markovian property and Chapman–Kolmogorov equation for the transition probabilities imply the semigroup property
\[
\Uptheta^{\varepsilon_1} \circ \Uptheta^{\varepsilon_1} = \Uptheta^{\varepsilon_1 + \varepsilon_2}, \quad \Uptheta^{\varepsilon} f(x) := \int_{\G} f(z) d\upmu_z^{\varepsilon}(w),
\]
characterizing $\Uptheta^{\varepsilon}$ as a Feller process. 
\par The generator $\mathcal{L}$ of this process is given by the strong limit
\[
\mathcal{L} f := \lim_{\varepsilon \downarrow 0} \nicefrac{1}{\varepsilon}\left( \Uptheta^{\varepsilon} f -f \right),
\]
on its domain of existence. The domain obviously contains $\mathcal{C}_{\sf fs}$. We also have 
\[
\Uptheta^{\varepsilon} u= e^{\varepsilon\mathcal{L}} u, \quad \text{for $u \in \mathsf{Dom}(\mathcal{L}) \supset  \mathcal{C}_{\sf fs}$},
\] 
and in particular, $\upmu_z^\varepsilon = e^{\varepsilon\mathcal{L}} \updelta_z$. 
\par  Thus in what follows we will focus on walks of the form  $e^{t\mathcal{L}} \updelta_z$ for a given operator $\mathcal{L}$. 
\begin{theorem}\label{thm:pls-walks}
	$\olric$ is well-defined for Markovian walks $e^{\varepsilon {\mathcal L}} \updelta_z$  when $\mathcal{L}$ is a good operator and it satisfies
	\begin{enumerate}
		\item $e^{\varepsilon \mathcal{L}}$ and $\text{-} \mathcal{L}$ satisfy rough comparison principles with range $0$;
	  \item  $\left| \mathcal{L} d(z,\cdot)(z) \right| \le C(z)$. 
	\end{enumerate}
\end{theorem}
\begin{proof}[\footnotesize \textbf{Proof}]
Since $\upmu_z^\varepsilon = e^{\varepsilon \mathcal{L}} \updelta_z$, it follows
	\[
	\mathcal{J}^\varepsilon = \sum\limits_{k=2}^\infty  \nicefrac{\varepsilon^{k-2}}{k!} \; \mathcal{L}^{k} = e^{\varepsilon \mathcal{L}} - \varepsilon \mathcal{L}- \mathbb{I},
	\]
	which clearly satisfies self-adjoint property and is a rough differential operator (all convergent series in $\mathcal{L}$, with no constant term, inherit these two properties form $\mathcal{L}$)
\par To verify a rough comparison principle for $\mathcal{J}^\varepsilon$, let $x$ be a global minimum of $f$. We can again assume $f(x)=0$ and $f\ge 0$. Then by (1),
\[
\mathcal{J}^\varepsilon  (f)(x) = e^{\varepsilon \mathcal{L}} f(x) - \varepsilon \mathcal{L}f(x) - \mathbb{I} f(x) \ge C_1 + \varepsilon C_2,
\]
\par Suppose $\left| \mathcal{L} \dist(z,\cdot)(z) \right| \le C_z $ is a bounded  then
\[
\left| \mathcal{J}^\varepsilon \dist(z,\cdot) (z) \right| \le \sum\limits_{k=2}^\infty  \nicefrac{1}{k!} \; C_z^{k} \le e^{C_z} .
\]
So $e^{\varepsilon A} \updelta_z$ is a pleasant walk and the existence of  $\olric$ follows from~Theorem~\ref{thm:pls-walk-0}. 
\end{proof}
\begin{corollary}\label{cor:mark-walk-2}Let 
let $\G$ be a locally $\dist$-finite graph. Then, $\olric$ is well-defined for Markovian walks $e^{\varepsilon \mathcal{L}} \updelta_z$  where $\mathcal{L}$ is a good operator that satisfies 
		\begin{enumerate}
		\item $e^{\varepsilon \mathcal{L}}$ satisfy rough rough comparison principle with range $0$;
		\item   $\mathcal{L}$ is a semi-local operator with range $R$;
		\item $\mathcal{L}$ satisfy rough comparison principle on $\Lip_1(\G)$ with range $2R$;
		\item $\left| \mathcal{L} \dist(z,\cdot)(w) \right| \le C(z)$ holds $\forall w \in \B_{2R}(z)$.
	\end{enumerate}
\end{corollary}
\begin{proof}[\footnotesize \textbf{Proof}]
	By Proposition~\ref{prop:st-comp}, we immediately deduce that the hypotheses in Theorem~\ref{thm:pls-walks} hold. 
\end{proof}
\subsubsection{\small \bf \textsf{Existence of $\olric$ for heat kernels}}\label{sec:ex-heat}
Now we will see the special case of the heat kernels as random walks i.e. we take the walks of the form $e^{\varepsilon \Delta}\updelta_z$.
\begin{theorem}\phantom{}
Suppose $\m$, $\dist$ and $\omega$ satisfy
\begin{enumerate}
	\item  $\m \in \ell^1(G)$;
	\item  $\m(z)^{-1}\sum_{w\sim z} \omega_{zw} \dist^2(z,w)$ is uniformly bounded for all $z \in G$;
	\item $\Deg_\omega$ is bounded on $T_{1}\left( \B_r(z)  \right)$ (combinatorial tubular neighborhood) for all $z \in G$ and $r>0$;
\end{enumerate}
Then $\olric$ is well-defined for the Markovian walk $e^{\varepsilon \Delta} \updelta_z$ and it coincides with ${^1}\olric$ curvature for the affine $1$-jet walk ${^1}\upmu_z^\varepsilon := \updelta_z + \varepsilon \Delta \updelta_z$.
\end{theorem}
\begin{proof}[\footnotesize \textbf{Proof}]
	Any constant re-scaling  $c\dist$ is again a distance, so we can assume 
	\[
		\m(z)^{-1}\sum_{w\sim z} \omega_{zw} \dist^2(z,w) \le 1;
	\]
namely, we can assume $\dist$ is intrinsic~\cite{Ke}. The hypotheses above imply that $\mathcal{C}_{\sf fs}(\G) \subset \LL^2(\G, \m)$ and furthermore,  $\Delta\restr_{\mathcal{C}_{\sf fs}(\G) }$ is essentially self adjoint~\cite{HKMW}. Obviously $\Delta 1 = 0$ so $\Delta$ is also a rough differential operator. 
\par By its definition, $\Delta$ satisfies rough comparison principles with zero range on all functions indeed at a global minimum $\Delta f(x) \ge 0$. 
Also  $-\Delta$ satisfies a rough comparison principle. The rough comparison principle for $\Delta$ will lead to a maximum principle which then is used to show $e^{\varepsilon \Delta}$ is positivity preserving hence, it also satisfies a rough comparison principle on $\LL^2$ functions. So with the hypotheses above, the operator $\Delta$ satisfies all the hypotheses of Theorem~\ref{thm:pls-walks}
\end{proof}
\begin{theorem}\label{thm:lap-wd-2}
Suppose $\dist= \dist_\eta$ for a secondary edge weight $\eta$ and $\left( \G, \dist_\eta \right)$ is metrically complete. Furthermore, assume
	\begin{enumerate}
		\item  $\m \in \ell^1(G)$;
		\item  $\m(z)^{-1}\sum_{w\sim z} \omega_{zw} \dist_\eta^2(z,w)$ is uniformly bounded for all $z \in G$.
	\end{enumerate}
Then, $\olric$ is well-defined for the Markovian walk $e^{\varepsilon \Delta} \updelta_z$ and it coincides with ${^1}\olric$ curvature for the affine walk ${^1}\upmu_z^\varepsilon := \updelta_z + \varepsilon \Delta \updelta_z$.
\end{theorem}
\begin{proof}[\footnotesize \textbf{Proof}]
	This follows immediately from \cite[Theorem 2]{HKMW} and our Theorem~\ref{thm:pls-walks}. 
\end{proof}
\begin{corollary}
Suppose $\dist = \dist_\eta$ is complete and 
	\begin{enumerate}
		\item  $\m \in \ell^1(G)$;
		\item  $\m(z)^{-1}\sum_{w\sim z} \omega_{zw} \eta_{zw}^2$ is uniformly bounded. 
	\end{enumerate}
then, the conclusion of the last two theorems hold. 
\end{corollary}
\begin{proof}[\footnotesize \textbf{Proof}]
	Item (2) here implies the item (2) in Theorem~\ref{thm:lap-wd-2}.  Also since $\dist_\eta$ is assumed to give rise to a metrically complete space, all metric balls are finite hence, (3) in  Theorem~\ref{thm:lap-wd-2} is automatically satisfied. 
\end{proof}
\begin{remark}\phantom{}\hspace{-1pt}\textit{The results in this section basically indicate that when the higher (than 1) order terms in a walk comprise a well-behaved zero mass signed measure, then they will not affect the $\olric$ that is given by a first derivative; namely, if the non locality is of second order in  $\varepsilon$, then the Ollivier-Ricci curvature exits. 
	 This seems very intuitive yet as we saw, this fact is very far from being trivial.}  \phantom{}
\end{remark}
%-----------------------------------------------
%-----------------------------------------------
\subsubsection{\small \bf  \textsf{Limit-free formulation}}\label{sec:lff}
\par A very useful limit-free formulation of discrete Ollivier-Ricci curvature and for $\upbeta$-walks was first observed in~\cite{MW}. Here, we first show a limit-free formulation in a more general framework of time-analytic local walks. This will lead to a limit-free formulation for pleasant and Markovian walks whose $1$-jets are local walks. 
\par We start with a continuous-time walk which is not necessarily Markovian.  Let $\upmu_z^\varepsilon$ be a local walk that is $\mathcal{C}^1$ in $\varepsilon$.
Define the (non-log-linear in $\varepsilon$) operators $\Uptheta^\varepsilon: \R^\mathcal{\G} \to \R^{\mathcal \G}$ via
\begin{align*}
	\Uptheta^\varepsilon(f)(z) &:= \int_{\Upomega_{z}}  f(w)  d\upmu_z^\varepsilon(w)  
	\\&= f(z) + \int_{\Upomega_{z}} \Big(f(w) - f(z)\Big)  d\upmu_z^\varepsilon(w)  
	\\&= f(z) + \sum_{\Upomega_{z}\smallsetminus\{z\}}  \left( f(w) - f(z) \right)  \upmu_z^\varepsilon(w) 
	\\&= \left( \mathbb{I} + \Psi^\varepsilon \right) f(z),
\end{align*}
where
\[
\Psi^\varepsilon f (z) := \sum_{\Upomega_{z}\smallsetminus\{z\}}  \Big( f(w) - f(z) \Big)  \upmu_z^\varepsilon(w);
\]
the above operators are well-defined by the finite-ness of $\Upomega_z$. $\Uptheta^\varepsilon$ is called the (non-linear) reverse generator for the random walk $\upmu_z^\varepsilon$. Notice if $\upmu_z^\varepsilon(w)$ is symmetric in $z$ and $w$, then $\Uptheta^\varepsilon$ is the (non-linear) generator of the random walk. 
\par As $\varepsilon$ varies $\Uptheta^\varepsilon$ produces a non-log-linear process on $\R^\G\to \R^\G$. The \emph{initial velocity} of this process is given by
\begin{align*}
\mathcal{L}f(z) &:= \lim_{\varepsilon \downarrow 0} \nicefrac{1}{\varepsilon} \left(  \Uptheta^\varepsilon -  \mathbb{I} \right)  f(z)
\\&= \lim_{\varepsilon \downarrow 0} \nicefrac{1}{\varepsilon} \sum_{\Upomega_z\smallsetminus\{z\}}  \left( f(w) - f(z) \right)  \upmu_z^\varepsilon(w) 
\\&= \nicefrac{d}{d\varepsilon}\restr_{\varepsilon=0} \Psi^\varepsilon f(x)
\\&= \sum_{\Upomega_z\smallsetminus\{z\}}  \left( f(w) - f(z) \right)  \nicefrac{d}{d\varepsilon} \restr_{\varepsilon=0} \upmu_z^\varepsilon(w) ;
\end{align*}
\emph{notice we did not call this the infinitesimal generator as $\Uptheta^\varepsilon$ is not necessarily a semigroup (hence, the term non-linear). } Due to the form it takes, one can think of $\mathcal{L}$ as a generalized Laplacian. 
\begin{theorem}[limit-free formulation]\label{thm:lff}
	Suppose $\upmu_z^\varepsilon$ is a time-analytic local walk. Then,
\begin{align}\label{eq:il-key-0}
\olric(x,y) =  \inf_{\substack{f: \mathcal{K}_{xy} \to \R \\ f \in \Lip(1)\\ \nabla_{xy} f = 1 }} \nabla_{yx}  \mathcal{L}(f);
\end{align}
where the finite set $\mathcal{K}_{xy}$ as in Definition~\ref{defn:loc-exc}. 
\end{theorem}
\begin{proof}[\footnotesize \textbf{Proof}]
	We essentially follow and adapt the proof of limit-free formulation in~\cite{MW}. 
\par By Kantorovich duality, one has
\begin{align*}
	\Was_1(\upmu_x^\varepsilon, \upmu_y^\varepsilon) = \sup_{\substack{f: \mathcal{K}_{xy} \to \R \\ f \in \Lip(1)}}  \sum f(z) \left( \upmu_x^\varepsilon(z) - \upmu_y^\varepsilon(z)  \right);
\end{align*}
hence,
\begin{align*}
	\Was_1(\upmu_x^\varepsilon, \upmu_y^\varepsilon) &= \sup_{\substack{f: \mathcal{K}_{xy} \to \R \\ f \in \Lip(1)}}  \left( f(x)  - f(y) \right)  + \Big(  \Psi^\varepsilon(f)(x) - \Psi^\varepsilon(f)(y) \Big) 
	\\&= \dist(x,y)  \sup_{\substack{f: \mathcal{K}_{xy} \to \R \\ f \in \Lip(1)}}  \nabla_{yx} f + \nabla_{yx} \Psi^\varepsilon(f);\notag
\end{align*}
and as a result,
\begin{align*}
	\olric_\varepsilon = 	1 - \nicefrac{\Was_1(\upmu_x^\varepsilon, \upmu_y^\varepsilon) }{\dist(x,y)} = \inf_{\substack{f: \mathcal{K}_{xy} \to \R \\ f \in \Lip(1)}}   (1 - \nabla_{xy} f ) + \nabla_{yx}  \Psi^\varepsilon(f).
\end{align*}
\par Adding more constraints, we deduce
\begin{align}\label{eq:il-key-1}
\lim_{\varepsilon \downarrow 0}	\nicefrac{1}{\varepsilon}	\olric_\varepsilon &=
	\lim_{\varepsilon \downarrow 0} \nicefrac{1}{\varepsilon}	\Big( 1 - \nicefrac{\Was_1(\upmu_x^\varepsilon, \upmu_y^\varepsilon) }{\dist(x,y)}  \Big)
\notag	\\&=
	\lim_{\varepsilon \downarrow 0}  \inf_{\substack{f: \mathcal{K}_{xy} \to \R \\ f \in \Lip(1)}}  \nicefrac{1}{\varepsilon} (1 - \nabla_{xy} f ) +  \nicefrac{1}{\varepsilon} \nabla_{yx}  \Psi^\varepsilon(f) \notag \\
	&\le \liminf_{\varepsilon \downarrow 0} \inf_{\substack{f: \mathcal{K}_{xy} \to \R \\ f \in \Lip(1)\\ \nabla_{xy} f = 1 }}  \nicefrac{1}{\varepsilon} \nabla_{yx}  \Psi^\varepsilon(f)
	\\& \le  \inf_{\substack{f: \mathcal{K}_{xy} \to \R \\ f \in \Lip(1)\\ \nabla_{xy} f = 1 }} \limsup_{\varepsilon \downarrow 0}  \nicefrac{1}{\varepsilon} \nabla_{yx}  \Psi^\varepsilon(f) \notag
	\\&=  \inf_{\substack{f: \mathcal{K}_{xy} \to \R \\ f \in \Lip(1)\\ \nabla_{xy} f = 1 }}  \nabla_{yx} \mathcal{L}(f) \notag.
\end{align}
notice all the $\lim_{\varepsilon \downarrow 0}$ exists. 
\par Suppose $f_\varepsilon$ is a minimizing sequence for  \eqref{eq:il-key-0}. The problem is invariant under translation of $f$ so we assume $f_\varepsilon(x) = 0$. So there exists a subsequence $\varepsilon_k$ such that $f_{\varepsilon_k}\upmu^{\varepsilon_k}$ converges to $f_0$ and we must have 
\[
\lim_{k \to \infty}(1 - \nabla_{xy} f_{\varepsilon_k} ) = 0,
\]
otherwise the \eqref{eq:il-key-1} would blow-up in the limit. $f_0$ is $1$-Lipschitz and $\nabla_{yx} f_0 = 1$. So using $f_0$ as a test function, we deduce
\begin{align*}
	\nicefrac{1}{\varepsilon}	\olric_\varepsilon &\ge  \inf_{\substack{f: \mathcal{K}_{xy} \to \R \\ f \in \Lip(1)\\ \nabla_{xy} f = 1 }}  \nicefrac{1}{\varepsilon} \nabla_{yx}  \Psi^\varepsilon\left(f \right).
\end{align*}
Therefore,
\begin{align}\label{eq:lff-2}
\lim_{\varepsilon \downarrow 0}		\nicefrac{1}{\varepsilon}	\olric_\varepsilon &\ge \lim_{\varepsilon \downarrow 0}  \inf_{\substack{f: \mathcal{K}_{xy} \to \R \\ f \in \Lip(1)\\ \nabla_{xy} f = 1 }}  \nicefrac{1}{\varepsilon} \nabla_{yx}  \Psi^\varepsilon(f) \notag
	\\&\ge  \inf_{\substack{f: \mathcal{K}_{xy} \to \R \\ f \in \Lip(1)\\ \nabla_{xy} f = 1 }} \liminf_{\varepsilon \downarrow 0} \nicefrac{1}{\varepsilon} \nabla_{yx}  \Psi^\varepsilon(f)
	\\&=  \inf_{\substack{f: \mathcal{K}_{xy} \to \R \\ f \in \Lip(1)\\ \nabla_{xy} f = 1 }} \nabla_{yx}  \mathcal{L}(f).\notag
\end{align}
The conclusion follows from \eqref{eq:il-key-1} and \eqref{eq:lff-2}. 
\end{proof}
\begin{remark}\phantom{}\hspace{-1pt}\textit{The proof of Theorem~\ref{thm:lff} only uses time-analyticity to ensure the existence of $\olric$ while the rest of the proof can be carried out verbatim for a local walk that is $\mathcal{C}^1$ in $\varepsilon$. }\phantom{}
\end{remark}
\begin{corollary}
	Suppose $\upmu_z^\varepsilon$ is a local walk with the properties
	\begin{enumerate}
		\item  $\upmu_z^\varepsilon$ is $\mathcal{C}^1$ in $\varepsilon$;
		\item $\olric(x,y)$ exists;
	\end{enumerate}
	then, the limit-free formulation \eqref{eq:il-key-0} holds true. 
\end{corollary}
\begin{example}\phantom\hspace{-2pt}For local $\uptheta$-walks, we deuce
\begin{align*}
  \olric(x,y) =  \nicelambda_3'(0)  \inf_{\substack{f: \mathcal{K}_{xy} \to \R \\ f \in \Lip(1)\\ \nabla_{xy} f = 1 }} \nabla_{yx}  \mathscr{L}(f),
\end{align*}
where
\[
\mathscr{L} f(z) := \sum_{w \in \mathcal{K}_{z} \smallsetminus \{z\}} \left(  f(w) - f(z)  \right) \varphi_2(z,w).
\]
\par As a result for $\upbeta$-walks, we get
\begin{align*}
	\olric(x,y) =  \inf_{\substack{f: \mathcal{K}_{xy} \to \R \\ f \in \Lip(1)\\ \nabla_{xy} f = 1 }} \nabla_{yx}  \Delta(f),
\end{align*}
which is the limit free formulation of~\cite{MW} and \cite{JM}. 
\par For $\upzeta$-walks, we get
\begin{align*}
	\olric(x,y) =  \inf_{\substack{f: \mathcal{K}_{xy} \to \R \\ f \in \Lip(1)\\ \nabla_{xy} f = 1 }} \nabla_{yx}  \Delta_{\sf n}(f);
\end{align*}
and for $\upxi$-walks we get
\begin{align*}
	\olric(x,y) =  \text{-} \nicefrac{1}{C}\inf_{\substack{f: \mathcal{K}_{xy} \to \R \\ f \in \Lip(1)\\ \nabla_{xy} f = 1 }} \nabla_{yx}  \Delta_{\upxi}(f),
\end{align*}
where $\Delta_{\upxi}$ is the Laplacian determined by $m=1$ and $\omega_{xy} = e^{-\dist(x,y)^p}$.
\end{example}
\begin{corollary}\label{cor:lff-pls-walk}
Suppose
\[
\upmu_z^\varepsilon = \updelta_z + \varepsilon \upmu_z + t^2\mathcal{R}_z,
\]
is a pleasant walk and $\upmu_z$ is of bounded support. then by Theorem~\ref{thm:lff}, we deduce
\begin{align*}
	\olric(x,y) =  \inf_{\substack{f: \mathcal{K}_{xy} \to \R \\ f \in \Lip(1)\\ \nabla_{xy} f = 1 }} \nabla_{yx}  \mathcal{L}(f),
\end{align*}
where $ \K_{xy} := \Upomega_x \cup \Upomega_y$ and
\begin{align*}
	\mathcal{L}f(z) := \sum_{z \in \Upomega_{z}\smallsetminus\{z\}}  \left( f(w) - f(z) \right)  \upmu_z(w).
\end{align*}
\end{corollary}
\begin{proof}[\footnotesize \textbf{Proof}]
Based on the proof of Theorem~\ref{thm:pls-walk-0}, we only need to work with the $1$-jet walk
\[
{^1}\upmu_z^\varepsilon = \updelta_z + \varepsilon \upmu_z.
\]
By Theorem~\ref{thm:lff}, the conclusion follows. 
\end{proof}
\begin{corollary}\label{cor:lff-mark-walk-1}
	Let $e^{\varepsilon \mathcal{L}} \updelta_z$ where $\mathcal{L}$ be a Markovian walk where $\mathcal{L}$ is a good operator and either
	\begin{enumerate}
		\item   $\mathcal{L}$ is a semi-local operator with range $R$;
	\item $e^{\varepsilon\mathcal{L}}$ and $\text{-} \mathcal{L}$ satisfy rough comparison principles with range $0$;
	\item  $\left| \mathcal{L} \dist(z,\cdot)(z) \right| \le C_3(z)$. 
\end{enumerate}
	then the limit-free formulation holds i.e.
	\begin{align*}
		\olric(x,y) =  \inf_{\substack{f: \mathcal{K}_{xy} \to \R \\ f \in \Lip(1)\\ \nabla_{xy} f = 1 }} \nabla_{yx}  \mathcal{L}(f),
	\end{align*}
in which $\K_{xy} := \B_{R+1}(x) \cup \B_{R+1}(y)$. 
\end{corollary}
\begin{proof}[\footnotesize \textbf{Proof}]
	Notice since $\mathcal{L}$ is semi-local with range $R$ then the support of the measure
	$
	\upmu_z := \mathcal{L}\updelta_z
	$
	is included in $\B_{R+1}(z)$. This means the $1$-jet walk is a local walk. The conclusion then follows from Corollary~\ref{cor:lff-pls-walk} and Theorem~\ref{thm:pls-walks}.
\end{proof}
\begin{corollary}\label{cor:lff-mark-walk-2}
		Let $e^{\varepsilon {\mathcal L}} \updelta_z$ where $\mathcal{L}$ be a Markovian walk where $\mathcal{L}$ is a good operator and either
			\begin{enumerate}
			\item $e^\mathcal{\varepsilon L}$ satisfy rough rough comparison principle with range $0$;
			\item   $\mathcal{L}$ is a semi-local operator with range $R$;
			\item $\mathcal{L}$ satisfy rough comparison principle on $\Lip_1(\G)$ with range $2R$;
			\item $\left| \mathcal{L} \dist(z,\cdot)(w) \right| \le C(z)$ holds $\forall w \in \B_{2R}(z)$.
		\end{enumerate}
then the limit-free formulation holds. 
\end{corollary}
\begin{proof}[\footnotesize \textbf{Proof}]
This directly follows form Corollary~\ref{cor:lff-mark-walk-1} and Corollary~\ref{cor:mark-walk-2}. 
\end{proof}
%-----------------------------------------------
%-----------------------------------------------
\subsubsection{\small \bf \textsf{The special case of multiply weighted Ollivier-Ricci}}
\par An special case if of course when the distance $\dist = \dist_\eta$ (see \eqref{eq:dist})  for some secondary edge weight $\eta$; see Definition~\ref{eq:dist}. As long as the $\left( \G, \dist_\eta \right)$ is a complete metric space, the theory developed here works without any changes. 
\par By the Hopf-Rinow type theorem in \cite{HKMW}, completeness of $\dist_\eta$ is equivalent to the fact that all metric balls are finite and to the fact that every bounded closed set is compact. This means $\G$ is locally $\dist_\eta$-finite. So, we indeed have shown the following.
\begin{theorem}
	Suppose $\eta$ is a secondary edge weight and $\G$ is locally $\dist_\eta$-finite. Then, all the results of this section -- thus far -- holds for $\dist = \dist_\eta$.
\end{theorem}
\begin{example}\phantom\label{ex:beta-lff}\hspace{-2pt}Setting $\eta = \omega$, then locally finiteness of $\G$ immediately implies the finiteness of all $\dist_\omega$ balls, hence, $\G$ is locally $\dist_\omega$-finite. Hence, all the results thus far apply. In particular, using $\upbeta$-walks, we retrieve the well-definition and limit-free formulation in \cite{MW}.  
\end{example}
\begin{remark}\phantom{}\hspace{-1pt}\textit{In practice, when dealing with complex multi-dimensional networks, there are a set of edge weights each representing a certain communication in the network. The general framework we presented in this article allows the random walks to depend on a primary edge weight $\omega_0$ and another family of edge weights $\omega_i$, $i\ge 1$ or even more general functions. Our constructions also allows for the distance function to be induced by a secondary edge weight $\upmu$ as we saw in above. This provides a lot of versatility for applying Ollivier-Ricci base methods to multi-dimensional networks. }\phantom{}
\end{remark}
%-----------------------------------------------
%-----------------------------------------------
\subsubsection{\small \bf \textsf{Operator theoretic Ollivier-Ricci curvature}}\label{sec:op-theo-ric}
The limit-free formulation enables us to generalize the Ollivier-Ricci curvature to be defined for operators instead of random walks.
\begin{definition}
Let $\mathcal{L}$ be an arbitrary operator
\[
\mathcal{L}: \R^G \supset \mathsf{Dom}(\mathcal{L}) \to \R^G.
\]
The corresponding operator-theoretic Ollivier-Ricci curvature is defined by
\[
\olric_{\mathcal{L}}(x,y) := \inf_{\substack{ f \in \mathcal{C}_{\sf fs}(\G)\\ f \in \Lip(1)\\ \nabla_{xy} f = 1 }} \nabla_{yx}  \mathcal{L}f,
\]
provided that $ \mathcal{C}_{\sf fs}(\G)  \subset  \mathsf{Dom}(\mathcal{L}) $, that is a very weak condition.  Notice we might get 
$
\olric_{\mathcal{L}}(x,y) = -\infty
$.
\end{definition}
\begin{remark}\phantom{}\hspace{-1pt}\textit{Suppose $\mathcal{L}$ is weakly of divergence form i.e.
	\[
	\int_{\G}\mathcal{L} f = \int_{\G}f, \quad \forall f\in \mathcal{C}_{\sf fs}(\G).
	\]
Then,  $\upmu_z^\varepsilon := e^{\varepsilon\mathcal{L}}\updelta_z$ is a continuous-time Markovian random walk which is in many cases not a local walk. This means the theory developed in these notes, does not ensure that Ollivier-Ricci curvature is well-defined for $\upmu_z^\varepsilon$; however, the operator-theoretic Ollivier-Ricci curvature  is well-defined even though it could be $-\infty$.}\phantom{}
\end{remark}
Recall $\mathcal{L}: \R^G \supset \mathsf{Dom}(\mathcal{L}) \to \R^G$ is said to be a bounded operator w.r.t. the sup-norm if there exits  $B>0$ such that
\[
\|\mathcal{L}f\|_{\sf sup} \le B \|f\|_{\sup},\quad  \forall f \in \mathcal{C}_{\sf fs}(\G).
\]
\begin{theorem}
Suppose $\mathcal{L}$ satisfies the following properties
\begin{enumerate}
	\item $\mathcal{L}$ is a rough differential operator;
	\item $\mathcal{L}$ satisfies a two-sided rough comparison principle (see Definition~\ref{def:two-sided});
\end{enumerate}
	Then, $\olric_{\mathcal{L}}(x,y)$  is finite. 
\end{theorem}
\begin{proof}[\footnotesize \textbf{Proof}]
		Since $\mathcal{L}$ is a rough differential operator, we can - with no loss of generality -- assume a given test function satisfies $f(x)=0$, $f(y)=\dist(x,y)$. 
\par The function
	\[
	\nicefrac{1}{2}\left( \dist(x,z) - f(z)   \right),
	\]
	is a nonnegative $1$-Lipschitz function attaining a global minimum at $x$ so by the rough maximum principle with range zero, one deuces
	\[
\mathcal{L}\left(  \dist(x,\cdot)\right)(x) - C_1(x) \le	\mathcal{L}f(x) \le \mathcal{L}\left(  \dist(x,\cdot)\right)(x) + C_1(x);
	\]
similarly
	\[
	\nicefrac{1}{2}\left( \dist(y,z) - f(z) + \dist(x,y)  \right),
	\]
	is nonnegative and has a global minimum at $y$ hence, from the hypotheses, we also get
	\[
\mathcal{L}\left(  \dist(x,\cdot)\right)(x) - C_1(y) \le	\mathcal{L}f(y) \le \mathcal{L}\left(   \dist(y,\cdot)\right)(y) + C_1(y).
	\]
therefore $
\nabla_{yx}  \mathcal{L}f 
$
is bounded and the infimum is finite. 
\end{proof}
\begin{theorem}
	$\olric_{\mathcal{L}}(x,y)$ is finite for all bounded semi-local rough differential operators $\mathcal{L}$.
\end{theorem}
\begin{proof}[\footnotesize \textbf{Proof}]
	Since $\mathcal{L}$ is a rough differential operator, it is translation invariant on $\mathcal{C}_{\sf fs}(\G)$ so we can assume our test functions satisfy $f(x) = 0$ and therefore they must satisfy $f(y) = \dist(x,y)$.  By semi-locality and by Lipschitz extension construction in the  proof of Theorem~\ref{thm:loc-dual} (see the bottom of page 19), we only need to find the infimum over test functions that are supported in 
$\B_{2s}(x)$ for $s= R + \dist(x,y)$ and $\|f\|_{\sf sup} \le 2s$. 
\par So for these test functions, by boundedness of $\mathcal{L}$, we deuce $\nabla_{yx}  \mathcal{L}f$ is bounded hence, its infimum is finite.   
\end{proof}
\begin{proposition}\label{prop:concav-op}
	For two operators $\mathcal{L}_1$ and $\mathcal{L}_2$, as in the above two theorems,
	\begin{align}\label{eq:conv-op}
	\olric_{t\mathcal{L}_1 + (1-t)\mathcal{L}_2} \ge t  \olric_{\mathcal{L}_1} + (1-t) \olric_{\mathcal{L}_2}.
	\end{align}
\end{proposition}
\begin{proof}[\footnotesize \textbf{Proof}]
	This is straightforward from the set-theoretic concavity of infimum. 
\end{proof}
\begin{remark}\phantom{}\hspace{-1pt}\textit{The concavity \eqref{eq:conv-op} is the infinitesimal (first order) version of the concavity~\eqref{eq:ric-conc} and holds for a wider class of operators than the generators of pleasant  Markovian walks; e.g. $\mathcal{L}$ need not be weakly of divergence type.}\phantom{} 
\end{remark}
%-----------------------------------------------
%-----------------------------------------------
\subsubsection{\small \bf \textsf{Lipschitz regularity}}
One important consequence of having a limit-free formulation at our disposal is that we can compute the continuous-time Ollivier-Ricci curvature as the optimal value of an LP problem; hence, we can for example show Lipschitz continuity of $\olric$ in terms of the distance and the $1$-jet of the random walk  for pleasant Markovian walks appearing in Corollary~\ref{cor:lff-mark-walk-1}. 
\par Based on the proof of Theorem~\ref{thm:pls-walk-0},  the Ollivier-Ricci curvature $\olric(x,y)$ for pleasant local walks
\[
\upmu_z^\varepsilon = \updelta_z + \varepsilon \upmu_z + \varepsilon^2\mathcal{R}_z,
\]
is a function of $\upmu_z$ and of the distance $\dist$ hence, we will write $\olric(\dist, \upmu)$ and we wish to establish Lipschitz regularity in the $\dist$ and $\upmu$ arguments. Let $\mathcal{M}_{\sf s}(\G)$ be the space of zero-mass signed measures on $\G$ equipped with the sup-norm; notice $\mathcal{M}_{\sf s}(\G)$ is a subspace of $\R^{\G \times \G}$. Also let $ \mathsf{D}_{\sf c}(\G)$ be the space of distances $\dist$ for which $\left( \G, \dist \right)$ is metrically complete also equipped with the sup-norm; again a subspace of  $\R^{\G \times \G}$. 
\par Set
\[
\mathcal{M}^{\K}_{\sf s}(\G):= \mathcal{M}_{\sf s}(\G) \cap \left\{  \upmu  \;\; \text{\textbrokenbar}\;\;  \mathrm{supp} (\upmu) \subset \K\right\}.
\]
\begin{theorem}[Lipschitz continuity]\label{thm:lip}
	Restricted to pleasant walks 
	\[
	\upmu_z^\varepsilon = \updelta_z + \varepsilon \upmu_z + \varepsilon^2\mathcal{R}_z,
	\]
	with $\upmu_x, \upmu_y \in \mathcal{M}^{\K}_{\sf s}(\G)$  for fixed $x$ and $y$ and some finite $\K$, the quantity $\olric(\dist, \upmu)$ is Lipschitz continuous in the $\dist$ and $\upmu_z$ arguments. 
\end{theorem}
\begin{proof}[\footnotesize \textbf{Proof}]
	By the limit-free formulation established in Corollary~\ref{cor:lff-pls-walk}, the identity
	\begin{align}\label{eq:lff}
		\olric(x,y) =  \inf_{\substack{f: \mathcal{K} \to \R \\ f \in \Lip(1)\\ \nabla_{xy} f = 1 }} \nabla_{yx}  \mathcal{L}(f),
	\end{align}
holds with
\begin{align*}
	\mathcal{L}f(z) = \sum_{w \in \K \smallsetminus\{z\}}  \left( f(w) - f(z) \right)  \upmu_z(w).
\end{align*}
\par Suppose $\left| \K \right| = N$, \eqref{eq:lff} is the optimal value in a linear programming problem in $\R^{N}$.   There are $2N(N-1)$  constraints.   Due to the form of $\mathcal{L}$, this problem is again translation invariant so adding another constraint $f(x) = \diam \K$, would ensure non-negativity of the feasible set. 
\par Based on the hypothesis, we only need to study the perturbation of $\upmu$ and $\dist$ on the set $\K$; so the problem at hand is a problem of the form
	\[
	\olric(\dist, \upmu): {\sf LP}_{\bf d} :=	\begin{cases} \min {\bf c}(\dist, \upmu) \cdot \hat{x} \\ {\bf a}\hat{x} \le {\bf b}(\dist)\\ \hat{x} \ge 0 \end{cases}, \quad {\bf d}(\dist,\upmu):= ({\bf a}, {\bf b}(\dist), {\bf c}(\dist,\upmu)),
	\]
	in which the arguments $\dist$ and $\upmu$ are functions with finite domains so we in particular have $\dist$ is bounded away from zero. 
\par First notice that both ${\bf c}(\dist,\upmu)$ and ${\bf b}(\dist)$ are Lipschitz functions of their arguments. This means to establish Lipschitz continuity  of  $\olric(\dist, \upmu)$ in $\dist$ and $\upmu$, we need to just establish the Lipschitz continuity of the optimal value function of ${\sf LP}_{\bf d} $ in terms of its vector variable ${\bf d}$; namely, we only need to show the optimal value function is Lipschitz in ${\bf a}$, ${\bf b}$ and $\bf{c}$. 
	\par Now, in order to establish the latter Lipschitz continuity by invoking Proposition~\ref{prop:lip-lp}, we need to ensure that the problem is feasible and is away from the boundary of the ill-posed region of parameters. Indeed, since we have a finite linear programming problem, this would automatically imply feasibility and distance to the ill-posed region for the dual problem as well. For the problem at hand these claims clearly hold since as long as $\dist$ stays within the space $\mathsf{D}_{\sf c}(\G)$, and as long as $\upmu$ is in $\mathcal{M}^{\K}_{\sf s}(\G)$, Theorem~\ref{thm:pls-walk-0} ensures the  existence of the optimal value. This means $\mathbf{d}$ is way from the boundary of ill-posed region for $\dist$ and $\upmu$ parameters in both primary and the dual problem. So, by Proposition~\ref{prop:lip-lp}, the Lipschitz continuity follows. 
\end{proof}
\begin{corollary}[Lipschitz continuity in the multi-weight case]\label{cor:lip-multi-weight}
	Suppose a pleasant local walk is a locally Lipschitz function of a primary edge weight $\omega_0$ and a set of other parameters $\omega_i$ ($i\ge 1$) as well as of the vertex measure $\m$. Also suppose the distance is induced by the secondary edge weight $\eta$ and all the metric balls are finite. Then for fixed $x$ and $y$, $\olric$ is locally Lipschitz as a function  of $\omega_0$, $\omega_i, i\ge 1$, $\m$ and $\eta$.
\end{corollary}
\begin{proof}[\footnotesize \textbf{Proof}]
	This directly follows from Theorem~\ref{thm:lip}. 
\end{proof}
%-----------------------------------------------
%-----------------------------------------------
\subsubsection{\small \bf \textsf{Scaling properties}}
The smooth Ricci bounds scale by $c^{-2}$ if we scale $g$ by $c^2$ or more precisely, say for a fixed vector $v$, $\Ric_{c^2g}(v,v) = \Ric_{g}(v,v)$. In the discrete setting, the scaling properties are interesting, to say the least and one needs to talk about both vertex weights and edge weights. 
\begin{proposition}
$\olric_\varepsilon$ and $\olric$ are invariant under the transformation $\dist \mapsto c \dist$  for $c>0$ provided $\olric$ is defined. 
\end{proposition}
\begin{proof}[\footnotesize \textbf{Proof}]
 The Wasserstein distance $\Was_1$ scales like $\dist$ so $\olric_\varepsilon$ is scale invariant. This means when the derivative at zero exists, it would also be scale invariant. 
\end{proof}
\begin{proposition}\label{prop:op-ric-scale}
	For the operator-theoretic Olivier-Ricci curvature, we also have
	\[
	\olric_{c\mathcal{L}} = c \; \olric_{\mathcal{L}}.
	\]
\end{proposition}
\begin{proof}[\footnotesize \textbf{Proof}]
This is straightforward from the definition. 
\end{proof}
\begin{corollary}
	For $\upbeta$-walks and under the transformation
	\[
	\omega \mapsto a\omega, \quad m \mapsto bm,
	\]
one has
	\[
	\olric (x,y) \mapsto \nicefrac{a}{b}\;  \olric (x,y).
	\]
	\end{corollary}
\begin{proof}[\footnotesize \textbf{Proof}]
	Directly follows from Proposition~\ref{prop:op-ric-scale} and the limit-free formulation; see Example~\ref{ex:beta-lff}; indeed, in this case $\mathcal{L} = \Delta$ with the transformation
	\[
	\Delta \mapsto \nicefrac{a}{b}\, \Delta.
	\]
\end{proof}
\begin{proposition}
	For time-affine walks, and under the transformation $\varepsilon \to c\varepsilon$, the transformation
	\[
	\olric_{c\varepsilon} (x,y) = \olric_{\varepsilon_i} + s_i \left(  c\varepsilon - \varepsilon_i \right), \quad \text{for}, \quad  \nicefrac{\varepsilon_i}{c} \le \varepsilon \le  \nicefrac{ \varepsilon_{i+1}}{c},
	\]
holds true,	where $\varepsilon_i$ are switching times and 
	\[
	s_i :=  \nicefrac{d^+}{d\varepsilon}\restr_{\varepsilon = \varepsilon_i}  \olric_{\varepsilon},
	\]
	are right derivatives at the switching times. 
\end{proposition}
\begin{proof}[\footnotesize \textbf{Proof}]
	\par We have shown in Theorem~\ref{thm:PL} that $\olric_{c\varepsilon} (x,y)$  is a piece-wise linear function of $\varepsilon$. So suppose the switching times $0 < \varepsilon_1 < \dots < \varepsilon_k$ are given; then
	\[
	\olric_{\varepsilon}(x,y) = \olric_{\varepsilon_i} + s_i \left(  \varepsilon - \varepsilon_i \right), \quad \text{for} \quad \varepsilon_i \le \varepsilon \le \varepsilon_{i+1},
	\]
So setting $\varepsilon = c \bar{\varepsilon}$, we get
	\[
\olric_{c\bar{\varepsilon}}(x,y) = \olric_{\varepsilon_i} + s_i \left(  c \bar{\varepsilon}- \varepsilon_i \right), \quad \text{for} \quad \varepsilon_i \le  c\bar{\varepsilon} \le \varepsilon_{i+1},
\]
which is the claimed relation. 
\end{proof}
\subsubsection{\small \bf \textsf{Switching times and bifurcation of optimal trajectories}}
The switching times in the function $\olric_\varepsilon$ for time-analytic local random walks indicate singularities in the discrete-time curvature function (as well as in the $\Was_1$). At these singularities, the maximizer in the Kantorovich dual formulation changes (from one vertex of the feasible convex polytope to another in the same facet).
\par Based on the Kantorvich-Rubinstein duality (see~\cite[Theorem 5.10, item 5.4 and item 5.16]{Vil}), we know the optimal plans are supported on the $\dist$-sub-differential of the maximizer $f$. So as the maximizer suddenly changes, the trajectories along which mass is transported also change. If we compare these trajectories to gradient flow of a potential function $f$ in the smooth setting, then the change in the trajectories compares to  a change in the qualitative behavior of solutions of an ode. So in a sense, at the switching times, the optimal trajectories bifurcate yet in a sudden discontinuous manner. It is worth mentioning that this is a very heuristic discussion since especially for the $\LL^1$ cost, the geometry of optimal transport trajectories become very irregular compared to the higher $\LL^p$ costs. However, in principle, what was said still makes sense.  
\par For time-analytic local walks, one can capture the infinitesimal effect of the said bifurcation on the discrete-time Ollivier-Ricci curvature  at the switching times by defining the set-valued bifurcating continuous-time Ollivier-Ricci curvature 
\[
\olric^\varepsilon_{\sf b}(x,y):= \left\{ \nicefrac{d^-}{d\varepsilon} \olric_\varepsilon(x,y) , \nicefrac{d^+}{d\varepsilon} \olric_\varepsilon(x,y) \right\},
\]
which is defined for all values of $\varepsilon$. Obviously $\olric(x,y) = \olric^0_{\sf b}(x,y)$. $\olric^\varepsilon_{\sf b}(x,y)$ is single-valued except possibly at the switching times. Notice for time-analytic local walks, $\olric^\varepsilon_{\sf b}(x,y)$ takes finite values.
\par Another useful quantity to investigate is 
\[
\delta \olric^\varepsilon_{\sf b}(x,y) := \left( \nicefrac{d^+}{d\varepsilon} - \nicefrac{d^-}{d\varepsilon} \right) \olric_\varepsilon(x,y).
\]
For example, for time-affine random walks, the concavity in $\olric_\varepsilon$ yields that  $\delta \olric^\varepsilon_{\sf b}(x,y)$ is everywhere zero except at the switching times where it is negative. 

\par Derivative of $\olric_\varepsilon$ measures the rate at which the walks converge in Wasserstein distance and a sudden change in this quantity indicates how -- on average -- strong the bifurcation of optimal trajectories is. For time-affine walks interestingly enough, at each bifurcation the said rate increases.  
%-----------------------------------------------
%-----------------------------------------------
\subsubsection{\small \bf  \textsf{Ollivier scalar curvature bounds}}
A natural definition of discrete-time scalar curvature for finite step walks is
\[
\oscal_{\varepsilon} (x) := \sum_{y \in \Upomega_x} \olric_{\varepsilon}(x,y),
\]
and the continuous-time scalar curvature is defied by
\[
\oscal (x) := \sum_{y \sim x} \olric(x,y).
\]
%-----------------------------------------------
%-----------------------------------------------
% Section: Edge Ricci curvatures
%-----------------------------------------------
%-----------------------------------------------
\section{Continuous-time discrete Ollivier-Ricci curvature flows}\label{sec:CTRF}
The continuous time Ollivier-Ricci flow was proposed in~\cite{Ol-survey} as a natural generalization of Ricci flow to the discrete setting using the Ollivier curvature; a $1$-parameter family of distances on a space $X$ (perhaps a graph) is said to be an Ollivier-Ricci flow whenever the continuous time equation
\begin{align*}
\nicefrac{d}{dt}\, \dist(x,y) = - \olric(x,y) \dist(x,y),
\end{align*}
is satisfied where $\olric$ at time $t$ is defined using a suitably chosen random walk on graphs and using distance $\dist$ at time $t$. For example, if the distance is the weighted distance (since the combinatorial one does not evolve), this flow becomes 
\begin{align*}
	\nicefrac{d}{dt}\, \dist_\omega(x,y) = - \olric(x,y) \dist_\omega(x,y);
\end{align*}
the issue with the latter is that in the setting of weighted graphs the solutions will not be unique.  To remedy this, one can instead consider the more local flow
\[
\dot{\omega} = - \olric \cdot \omega,
\]
as the Ricci flow in the singly weighted setting. Another approach would be to take a fusion of the last two and define Ricci flow such as (but not exclusively)
\[
\dot{\omega} = - \olric \cdot \dist_\omega.
\]
\par In these notes, we have developed a general theory for Olivier-Ricci curvature and have seen many important properties such as the Lipschitz continuity for $\olric$; so it is natural to also explore the Ricci flow in this generalized context. In an upcoming work, we explore the discrete-time Ollivier-Ricci flows and in these notes, we will only consider the continuous-time version. 
\par In below, we consider general curvature flows corresponding to Ollivier-Ricci curvature however we will only consider finite graphs. 
\begin{definition}[Generalized Ollivier-Ricci curvature flows]
Let $\upmu_z^\varepsilon$ be a pleasant local walk that is locally Lipschitz in  $\omega$. A $1$-parameter family 
\[
\G(t) := \left( \G, \m(t), \omega(t), \dist(t)\right),
\]
of quadruples is called a generalized Ollivier-Ricci flow if it is a solution to the ODE system
\begin{align}\label{eq:ORF}
{\sf ORF}_{f,g,h} :
\begin{cases}
	\dot{\omega} = f(t, \omega, \dist, \m, \olric )\\
	\dot{\dist} =  g\left(t, \omega, \dist, \m, \olric \right)\\
	\dot{\m} = h\left(t, \omega, \dist, \m, \oscal \right) 
\end{cases}
\end{align}
\end{definition}
\begin{theorem}[Short time existence and uniqueness]\label{thm:EUR}
Let $\G$ be finite. Suppose $f,g$ and $h$ are locally uniformly Lipschitz in $t$ and locally Lipschitz in other arguments. Starting from an initial condition $\left( G, \m(0), \omega(0), \dist(0)\right)$ with $\dist(0) \in \mathrm{int}\left(  \mathsf{D}_{\sf c}(\G) \right)$ and with $\mathrm{supp}(\m)= \G$.   Then, there exists a unique solution to ${\sf ORF}_{f,g},h$ for a short time. The solution is at least $\mathcal{C}^{1,1}$ in time $t$ and is Lipschitz in the other arguments. 
\end{theorem}
\begin{proof}[\footnotesize \textbf{Proof}]
	Unlike the smooth Ricci flow which is a weakly parabolic pde, the continuous time Ollivier Ricci flow is indeed a (not necessarily autonomous) ODE system with phase space
	\[
	\R^{|E|} \times \mathsf{D}(\G) \times \R^{|V|} \times \R^{|V|(|V|-1)}\subset \R^{|E| + 2|V|^2 - |V|}.
	\]
where $\mathsf{D}(\G)$ is the space of distance functions which is a subspace of $\R^{|V|(|V|-1)}$. 
\par In the interior of the phase space -- by the standard ode theory (Picard - Lindel\"of theorem) -- the existence of solutions follow from the continuity of  the RHS while the uniqueness needs locally uniformly Lipschitz continuity of the RHS in $t$ and locally Lipschitz continuity in other variables; e.g. see~\cite[Theorem 2.2]{Tes}.
	\par Local Lipschitz continuity of $\olric$ was established in Theorem~\ref{thm:lip} (recall $\G$ is finite here so all walks are local) so this readily implies the RHS of \eqref{eq:ORF} is locally Lipschitz in $\dist$ and $\omega$. Also from the hypothesis, RHS is uniformly locally Lispchitz in $t$ so the conclusion follows. 
\end{proof}
\begin{theorem}[long-time solution]
	Starting from $\G(0)$ in the interior of the phase space, the generalized Ollivier-Ricci curvature flow ${\sf ORF}_{f,g,h}$ admits a unique maximal solution. Indeed the flow can be continued as long as $\omega, \m >0$ and $\dist$ is in the interior of $\mathsf{D}(\G)$. 
\end{theorem}
%-----------------------------------------------
%-----------------------------------------------
\subsubsection{\small \bf  \textsf{Classic Ollivier-Ricci flows}}
Set $\dist=\dist_\omega$. Consider the well-posed Ollivier-Ricci flow given by
\begin{align}\label{eq:corf}
\begin{cases}
	\dot{\omega} = -\olric \cdot \omega\\
	\dot{\m} = -h\left(\oscal\right) \cdot \m
\end{cases},
\end{align}
so when $h\equiv 0$, we just get a flow of $\omega$ and when $h = \mathrm{id}$, we get an evolution on vertex weights modeled on the smooth case. 
\begin{corollary}
Set $\dist = \dist_\omega$ and suppose $\upmu_z^\varepsilon$ is a random walk that is locally Lipschitz in $\omega$, then the Ollivier-Ricci flow \eqref{eq:corf}  starting with $\omega(0)>0$, has a unique solution. The maximal solution exists as long as $\omega, \m >0$. 
\end{corollary}
		\begin{proof}[\footnotesize \textbf{Proof}]
		This is a direct consequence of Corollary~\ref{cor:lip-multi-weight} and Theorem~\ref{thm:EUR}. 
	\end{proof}
	\renewcommand{\baselinestretch}{1.2}
%%%%%%%%%%%%%%
%%%%%%%%%%%%%%
\vspace{10pt}
%%%%%%%%%%%%%%
%%%%%%%%%%%%%%
\end{document}